\documentclass[11pt,french,english,a4paper]{amsart}
\usepackage[utf8]{inputenc}
\usepackage[T1]{fontenc}
\usepackage[french]{babel}
\usepackage{verbatim}

\usepackage{amsmath,amssymb}
\usepackage{amsthm}
\usepackage{lmodern}
\usepackage{latexsym}
\usepackage{amsfonts}
\usepackage{mathrsfs}
\usepackage[all]{xy}
\usepackage{bm}
\usepackage{yfonts}

\usepackage{hyperref}
\hypersetup{pdftoolbar=true, pdftitle={Proj_Ellip}, pdffitwindow=true, colorlinks=true, citecolor=green, filecolor=black, linkcolor=red, urlcolor=blue, hypertexnames=false}

\numberwithin{equation}{section}

\usepackage[top=4cm, bottom=3.4cm, left=3.3cm , right=3.3cm]{geometry}

\newtheorem{theorem}{Theorem}[section]
\newtheorem{proposition}[theorem]{Proposition}
\newtheorem{lemme}[theorem]{Lemma}
\newtheorem{mainthm}{Theorem}
\newtheorem{Hyp}{Hypothesis}
\newtheorem{cor}[mainthm]{Corollary}

\usepackage{enumerate}

\usepackage{fge}

\theoremstyle{definition}
\newtheorem{defi}[theorem]{Definition}
\newtheorem{remq}[theorem]{Remark}

\usepackage[mathcal]{euscript}

\newcommand{\Cbb}{\mathbb{C}}
\newcommand{\Mscr}{\mathscr{M}}

\newcommand{\ifm}{\mathrm{if}}

\newcommand{\Fbf}{\mathbf{F}}

\newcommand{\cond}{\mathrm{cond}}
\newcommand{\PGL}{\mathrm{PGL}}
\newcommand{\GL}{\mathrm{GL}}

\newcommand{\chibar}{\overline{\chi}}
\newcommand{\Brm}{\mathrm{\mathbf{B}}}
\newcommand{\Trm}{\mathrm{T}}
\newcommand{\Nrm}{\mathrm{\mathbf{N}}}

\newcommand{\Fcal}{\mathcal{F}}

\newcommand{\Tr}{\mathrm{Tr}}

\newcommand{\Lcal}{\mathcal{L}}

\newcommand{\amf}{\textfrak{a}}

\newcommand{\Ecal}{\mathcal{E}}

\newcommand{\Abf}{\mathbf{A}}
\newcommand{\omegabar}{\overline{\omega}}
\newcommand{\Acal}{\mathcal{A}}
\newcommand{\Abb}{\mathbb{A}}
\newcommand{\sbar}{\overline{s}}
\newcommand{\Oscr}{\mathscr{O}}

\newcommand{\Cscr}{\mathscr{C}}
\newcommand{\Dscr}{\mathscr{D}}

\newcommand{\Zrm}{\mathrm{\mathbf{Z}}}

\newcommand{\Prm}{\mathrm{\mathbf{P}}}
\newcommand{\Bscr}{\mathscr{B}}
\newcommand{\Nm}{\mathrm{\mathbf{N}}}

\newcommand{\Grm}{\mathrm{\mathbf{G}}}
\newcommand{\Arm}{\mathrm{A}}

\newcommand{\Xbf}{\mathbf{X}}
\newcommand{\Rbb}{\mathbb{R}}
\newcommand{\mmf}{\textfrak{m}}
\newcommand{\lmf}{\textfrak{l}}
\newcommand{\Qbb}{\mathbb{Q}}

\newcommand{\Pscr}{\mathscr{P}}

\newcommand{\Bm}{\mathrm{\mathbf{B}}}
\newcommand{\reg}{\mathrm{reg}}
\newcommand{\Krm}{\mathrm{\mathbf{K}}}
\newcommand{\Nbb}{\mathbb{N}}
\newcommand{\Exp}{\Ecal\mathrm{x}}
\newcommand{\Erm}{\mathrm{E}}
\newcommand{\Lrm}{\mathrm{L}}
\newcommand{\X}{\Xbf_{\PGL_2}}
\newcommand{\pmf}{\textfrak{p}}
\newcommand{\qmf}{\textfrak{q}}
\newcommand{\F}{\mathrm{\mathbf{F}}}

\newcommand{\crm}{\mathrm{c}}
\newcommand{\Wcal}{\mathcal{W}}
\newcommand{\Ocal}{\mathcal{O}}

\newcommand{\Zcal}{\mathcal{Z}}

\newcommand{\Lscr}{\mathscr{L}}

\newcommand{\Vcal}{\mathcal{V}}

\newcommand{\Gscr}{\mathscr{G}}
\newcommand{\Wrm}{\mathrm{W}}
\newcommand{\Nscr}{\mathscr{N}}

\title{Periods and Reciprocity II}
\author{Raphaël Zacharias}
\date{\today}

\begin{document}

\maketitle

%%%%%%%%%%%%%%%%%%%%%%%%%%%%%%%%%%%%%%%%%%%%%%%%%%%%%%
%%%%%%%%%%%%%%%%%%%%%%%%%%%%%%%%%%%%%%%%%%%%%%%%%%%%%%
%%%%%%%%%%%%%%%%%%%%%%%%%%%%%%%%%%%%%%%%%%%%%%%%%%%%%%
%%%%%%%%%%%%%%%%%%%%%%%%%%%%%%%%%%%%%%%%%%%%%%%%%%%%%%
%%%%%					ABSTRACT
%%%%%%%%%%%%%%%%%%%%%%%%%%%%%%%%%%%%%%%%%%%%%%%%%%%%%%
%%%%%%%%%%%%%%%%%%%%%%%%%%%%%%%%%%%%%%%%%%%%%%%%%%%%%%
%%%%%%%%%%%%%%%%%%%%%%%%%%%%%%%%%%%%%%%%%%%%%%%%%%%%%%
%%%%%%%%%%%%%%%%%%%%%%%%%%%%%%%%%%%%%%%%%%%%%%%%%%%%%%
\begin{abstract}
Let $\F$ be a number field with adele ring $\Abb_\F$, $\textfrak{q},\textfrak{l}$ two coprime integral ideals with $\qmf$ squarefree and $\pi_1,\pi_2$ two fixed unitary automorphic representations of $\PGL_2(\Abb_\F)$ unramified at all finite places. In this paper, we use regularized integrals to obtain a formula that links the first moment of $L(\pi\otimes\pi_1\otimes\pi_2,\tfrac{1}{2})$ twisted by the Hecke eigenvalues $\lambda_\pi (\lmf)$, where $\pi$ runs through unitary automorphic representations of $\PGL_2(\Abb_\F)$ with conductor dividing $\qmf$, with some spectral expansion of periods over representations of conductor dividing $\lmf$. In the special case where $\pi_1=\pi_2=\sigma$, this formula becomes a reciprocity relation between moments of $L$-functions. As applications, we obtain a subconvex estimate in the level aspect for the central value of the triple product $L(\pi\otimes\pi_1\otimes\pi_2,\tfrac{1}{2})$ and a simultaneous non-vanishing result for $L(\mathrm{Sym}^2(\sigma)\otimes \pi,\tfrac{1}{2})$ and $L(\pi,\tfrac{1}{2})$.
\end{abstract}
\tableofcontents

%%%%%%%%%%%%%%%%%%%%%%%%%%%%%%%%%%%%%%%%%%%%%%%%%%%%%%
%%%%%%%%%%%%%%%%%%%%%%%%%%%%%%%%%%%%%%%%%%%%%%%%%%%%%%
%%%%%%%%%%%%%%%%%%%%%%%%%%%%%%%%%%%%%%%%%%%%%%%%%%%%%%
%%%%%%%%%%%%%%%%%%%%%%%%%%%%%%%%%%%%%%%%%%%%%%%%%%%%%%
%%%%%%%%%%%%%%%%%%%%%%%%%%%%%%%%%%%%%%%%%%%%%%%%%%%%%%
%%%%%									INTRODUCTION
%%%%%%%%%%%%%%%%%%%%%%%%%%%%%%%%%%%%%%%%%%%%%%%%%%%%%%
%%%%%%%%%%%%%%%%%%%%%%%%%%%%%%%%%%%%%%%%%%%%%%%%%%%%%%
%%%%%%%%%%%%%%%%%%%%%%%%%%%%%%%%%%%%%%%%%%%%%%%%%%%%%%
%%%%%%%%%%%%%%%%%%%%%%%%%%%%%%%%%%%%%%%%%%%%%%%%%%%%%%
%%%%%%%%%%%%%%%%%%%%%%%%%%%%%%%%%%%%%%%%%%%%%%%%%%%%%%
\section{Introduction}
 Let $q,\ell$ be two coprime integers. Reciprocity Formulae for moments of automorphic $L$-functions have received a considerable attention in the past few years \cite{conrey,young,bettin,blomerspectral,blomer4,andersen}. These are remarkable results for at least two reasons. The first is their applications to the subconvexity problem and the non-vanishing of $L$-functions. The second, more conceptual, is that it connects objects, at least in the $\GL_2$ case, that have no a priori reason to be linked, like the different hyperbolic surfaces $\Gamma_0(q)\setminus \mathbb{H}$ and $\Gamma_0(\ell)\setminus \mathbb{H}$.

%%%%%%%%%%%%%%%%%%%%%%%%%%%%%%%%%%%%%%%%%%%%%%%%%%%%%%%
%%%%%%%%%%%%%%%%%%%%%%%%%%%%%%%%%%%%%%%%%%%%%%%%%%%%%%%
%%%%%%%%%%%%%%%%%%%%%%%%%%%%%%%%%%%%%%%%%%%%%%%%%%%%%%%
%%%%%%%%%%%%%%%%%%%%%%%%%%%%%%%%%%%%%%%%%%%%%%%%%%%%%%%
%%%%%%				The Idea of the Paper
%%%%%%%%%%%%%%%%%%%%%%%%%%%%%%%%%%%%%%%%%%%%%%%%%%%%%%%
%%%%%%%%%%%%%%%%%%%%%%%%%%%%%%%%%%%%%%%%%%%%%%%%%%%%%%%
%%%%%%%%%%%%%%%%%%%%%%%%%%%%%%%%%%%%%%%%%%%%%%%%%%%%%%%
%%%%%%%%%%%%%%%%%%%%%%%%%%%%%%%%%%%%%%%%%%%%%%%%%%%%%%%
\subsection{The idea of the paper}\label{Section1}
\noindent Let $\F$ be a number field and denote by $\Abb_\F$ its adele ring and by $\Ocal_\F$ its ring of integers. Let $\pmf,\qmf$ two prime ideals of $\Ocal_\F$ of norm $p$ and $q$ respectively and $\pi_1,\pi_2$ be two unitary automorphic representations of $\PGL_2(\Abb_\F)$ which are unramified at all finite places and with $\pi_1$ cuspidal.
In our previous paper \cite{raphael}, we proved, using a method related to adelic periods of automorphic forms, that the work of Andersen and Kiral \cite{andersen} holds more generally over any number fields and for the first moment of a triple product $L$-function. We obtained a relation of the shape
\begin{equation}\label{Equ1}
\begin{split}
\mathrm{M}(\pi_1,\pi_2,\qmf,\pmf) := & \ \sum_{\substack{\pi \\ \cond(\pi)|\qmf}}L(\pi\otimes\pi_1 \otimes\pi_2,\tfrac{1}{2})\lambda_\pi(\pmf)\varepsilon_\pi(\qmf) \\  & \ \rightsquigarrow \left(\frac{q}{p}\right)^{1/2}\sum_{\substack{\pi  \\ \cond(\pi)|\pmf}}L(\pi\otimes\pi_1\otimes\pi_2,\tfrac{1}{2})\lambda_\pi(\qmf)\varepsilon_\pi(\pmf),
\end{split}
\end{equation}
where $\lambda_\pi(\pmf)$ is the eigenvalue of the Hecke operator $\Trm_\pmf$ and $\varepsilon_\pi(\qmf)\in\{-1,+1\}$ denotes the Atkin-Lehner eigenvalue at $\qmf$. We recover the result of Andersen and Kiral in the particular case where $\F=\mathbb{Q}$ and $\pi_2=1\boxplus 1$ is the Eisenstein series induced from two trivial characters.

\vspace{0.1cm}

In this paper, we study a similar reciprocity relation, but with two mains novelties :
\begin{enumerate}[(1)]
\item Firstly, we managed to cancel the root number $\varepsilon_\pi(\qmf)$ on the lefthand side of \eqref{Equ1}. This is an important structural feature, because it allows to use positivity arguments, typically for the subconvexity problem in Theorem \ref{Theorem2}.
\item Secondly, using a regularized version of Plancherel formula introduced for the first time by Zagier \cite{zagier} and then developped adelically by Michel and Venkatesh in \cite{subconvexity}, we are also able to include the case $\pi_1=\pi_2=1\boxplus 1$, giving in particular a reciprocity statement for a fourth moment of $\GL_2$ automorphic $L$-functions, as in \cite{blomerspectral}. 
\end{enumerate}
We mention that if we do not include the root number in our moment, we cannot expect a perfect symmetry between these two moments, as in \eqref{Equ1}. The principal reason is that if we estimate \og trivially \fg{} the righthand side, we see that modulo the Ramanujan-Petersson Conjecture, we have
$$\mathrm{M}(\pi_1,\pi_2,\qmf,\pmf) \ll (pq)^{1/2}.$$
Hence taking $\pmf=1$, we see that $\mathrm{M}(\pi_1,\pi_2,\qmf,1)$ has size $q^{1/2}$ while for the same moment without the root number, we should have a main term of size $q$. We see in particular that the root number has a square root cancellation effect. We can summarize this by saying that
\begin{alignat*}{1}
\mathrm{root \ number } & \ \leadsto \mathrm{perfect \ symmetry}  \\
 & \\ 
\mathrm{no \ root \ number} & \ \leadsto \mathrm{main \ term}
\end{alignat*}
(see also \cite{blomer4} where this is the key observation).
This phenomenon can be read directly in some periods underlying the moments of $L$-functions. For example, the period giving \eqref{Equ1} is (see \cite[(3.1)]{raphael}) 
$$\int_{\PGL_2(\F)\setminus \PGL_2(\Abb_\F)} \varphi_1^{\pmf\qmf}\varphi_2^\pmf \overline{\varphi_1\varphi_2^\qmf}=\left\langle \varphi_1^{\pmf\qmf}\varphi_2^{\pmf},\varphi_1\varphi_2^\qmf\right\rangle,$$
where for any finite place $v$, $\varphi_i^v$ denotes the right translate of $\varphi_i$ by the matrix $\left(\begin{smallmatrix}
1 & \\ & \varpi_v
\end{smallmatrix}\right)$ and $\varpi_v$ is a uniformizer for the completed field $\F_{v}$. This is a symmetric expression in $(\pmf,\qmf)$ (modulo complex conjugate). If we take $\pmf=1$,  the integrand is not positive and in fact oscillates because the translation is not on the same vector in both sides of the inner product; this oscillation is exactly the root number when we pass to moments of $L$-functions via Plancherel formula. 

\vspace{0.1cm}

If we want to avoid the root number on the left of \eqref{Equ1}, we need to consider periods such that when we set $\pmf=1$, we get a positive integrand. Ignoring convergence problems due the Eisenstein case, the simplest one is
\begin{equation}\label{Type1}
\int_{\PGL_2(\F)\setminus \PGL_2(\Abb_\F)} \varphi_1^{\pmf\qmf}\varphi_2^{\pmf}\overline{\varphi_1^\qmf\varphi_2}=\left\langle \varphi_1^{\pmf\qmf}\varphi_2^{\pmf},\varphi_1^\qmf\varphi_2\right\rangle.
\end{equation}
Observe that this integral is no more symmetric, except when $\varphi_1=\varphi_2$. Assuming that $\varphi_1,\varphi_2$ are unramified at the finite places, the last inner product is roughly equal to 
$$\frac{1}{p^{1/2}}\sum_{\substack{\pi \\ \crm(\pi) | \qmf}}\lambda_\pi(\pmf)\sum_{\psi\in\Bscr(\pi,\qmf)}\left| \left\langle \psi,\varphi_1^\qmf \varphi_2 \right\rangle\right|^2,$$
where $\Bscr(\pi,\qmf)$ denotes an orthonormal basis of $\Krm_0(\qmf)$-invariant vectors in $\pi$ (see Section \ref{Sectionsubgroups} for notations). Using integral representations of $L$-functions recalled in Section \ref{SectionRankin}, this is clearly related to the moment of $L(\pi\otimes\pi_1\otimes\pi_2,\tfrac{1}{2})$. Now using the trick of \cite{subconvexity}, we observe that the inner product can be also written as
\begin{equation}\label{PeriodExpansion}
\left\langle \varphi_1^{\pmf\qmf}\overline{\varphi}_1^\qmf, \varphi_2\overline{\varphi}_2^{\pmf}\right\rangle \asymp \frac{1}{q^{1/2}}\sum_{\crm(\pi)|\pmf}\lambda_\pi(\qmf)\sum_{\psi\in\Bscr(\pi,\pmf)}\left\langle \varphi_1^\pmf\overline{\varphi}_1, \psi\right\rangle\left\langle \psi, \varphi_2\overline{\varphi}_2^\pmf\right\rangle
\end{equation}
plus extra terms coming from regularization if one of the $\pi_i$'s is Eisenstein. The righthand side becomes a moment of $L$-functions in the special case where $\pi_1=\pi_2$ and $\varphi_1=\varphi_2$, giving a reciprocity formula for the fourth moment of $L(\pi,\tfrac{1}{2})$ in the case $\pi_1=\pi_2=1\boxplus 1$, or for $L(\mathrm{Sym}^2(\sigma)\otimes \pi,\frac{1}{2})L(\pi,\tfrac{1}{2})$ when $\pi_1=\pi_2=\sigma$ is cuspidal (c.f. Theorem \ref{Theorem3}). Observe that the latter one is a special case of \cite{blomerspectral} where the $\GL_3$ cusp form is the symmetric square lift of $\sigma$.

\vspace{0.1cm}

For other configurations of the $\pi_i$'s, \eqref{PeriodExpansion} is just a spectral expansion of periods (c.f. \eqref{ReciprocityRelation} for the reciprocity relation of periods). However, we can still estimate this last quantity in great generality, even if it is not directly connected to $L$-functions (c.f. \eqref{FinalBound1}). This is quite satisfactory, especially when we have applications in minds.

%%%%%%%%%%%%%%%%%%%%%%%%%%%%%%%%%%%%%%%%%%%%%%%%%%%%%%%
%%%%%%%%%%%%%%%%%%%%%%%%%%%%%%%%%%%%%%%%%%%%%%%%%%%%%%%
%%%%%%%%%%%%%%%%%%%%%%%%%%%%%%%%%%%%%%%%%%%%%%%%%%%%%%%
%%%%%%%%%%%%%%%%%%%%%%%%%%%%%%%%%%%%%%%%%%%%%%%%%%%%%%%
%%%%%							STATEMENT OF RESULTS
%%%%%%%%%%%%%%%%%%%%%%%%%%%%%%%%%%%%%%%%%%%%%%%%%%%%%%%
%%%%%%%%%%%%%%%%%%%%%%%%%%%%%%%%%%%%%%%%%%%%%%%%%%%%%%%
%%%%%%%%%%%%%%%%%%%%%%%%%%%%%%%%%%%%%%%%%%%%%%%%%%%%%%%
%%%%%%%%%%%%%%%%%%%%%%%%%%%%%%%%%%%%%%%%%%%%%%%%%%%%%%%
\subsection{Statement of results}

In order to state properly the results, we fix $\pi_1,\pi_2$ two unitary automorphic representations of $\PGL_2(\Abb_\F)$ and let $\qmf$ be a squarefree ideal of $\Ocal_\F$ with norm $q$. We will assume most of the time that $\pi_1$ and $\pi_2$ satisfy the following hypothesis :
%%%%%%%%%%%%%%%%%%%%%%%%%%%%%%%%%%%%%%%%%%%%%%%%%%%%%%
%%%%%%%%%%%%%%%%%%%%%%%%%%%%%%%%%%%%%%%%%%%%%%%%%%%%%%
%%%%%				HYPOTHESIS 1
%%%%%%%%%%%%%%%%%%%%%%%%%%%%%%%%%%%%%%%%%%%%%%%%%%%%%%
%%%%%%%%%%%%%%%%%%%%%%%%%%%%%%%%%%%%%%%%%%%%%%%%%%%%%%
\begin{Hyp}\label{Hyp1}
\begin{enumerate}[$(1)$] 
\item $\pi_1$ and $\pi_2$ are unramified at the finite places \footnote{This condition could be easily removed, especially for Theorem \ref{Theorem2}. If $\pi_i$ has finite conductor $\textfrak{n}_i$, then the spectral expansion \eqref{DefinitionMoment1} should be over representations of conductor dividing $\textfrak{n}_1\textfrak{n}_2\qmf$.};
\item $\pi_1$ and $\pi_2$ are tempered at the finite places;
\item If one of the $\pi_i$'s is Eisenstein, then it is of the form $|\cdot|^{it}\boxplus |\cdot|^{-it}$ for some $t\in \Rbb$.
\end{enumerate}
\end{Hyp}

%%%%%%%%%%%%%%%%%%%%%%%%%%%%%%%%%%%%%%%%%%%%%%%%%%%%%%%
%%%%%%%%%%%%%%%%%%%%%%%%%%%%%%%%%%%%%%%%%%%%%%%%%%%%%%%
%%%%%%%%%%%%%%%%%%%%%%%%%%%%%%%%%%%%%%%%%%%%%%%%%%%%%%%
%%%%%				AN UPPER BOUND FOR THE TWISTED MOMENT
%%%%%%%%%%%%%%%%%%%%%%%%%%%%%%%%%%%%%%%%%%%%%%%%%%%%%%%
%%%%%%%%%%%%%%%%%%%%%%%%%%%%%%%%%%%%%%%%%%%%%%%%%%%%%%%
%%%%%%%%%%%%%%%%%%%%%%%%%%%%%%%%%%%%%%%%%%%%%%%%%%%%%%%
\subsubsection{An upper bound for the twisted moment with application to subconvexity}
Let $\lmf$ be an integral ideal of norm $\ell$. We define the following quantities : first the cuspidal part by
\begin{equation}\label{CuspidalPart}
\mathscr{C}(\pi_1,\pi_2,\qmf,\lmf):= C\sum_{\substack{\pi \ \mathrm{cuspidal} \\ \crm(\pi)|\qmf}}\lambda_\pi(\lmf)\frac{L(\pi\otimes\pi_1\otimes\pi_2,\tfrac{1}{2})}{\Lambda(\pi,\mathrm{Ad},1)}f(\pi_\infty)\Lcal(\pi,\qmf),
\end{equation}
where $C$ is a positive constant depending only on $\F$ and the nature of the three representations $\pi,\pi_1,\pi_2$ (c.f. \eqref{FactorizationLcal}-\eqref{ValueC}). The two other factors $f(\pi_\infty)$ and $\Lcal(\pi,\qmf)$ can be described as follows.
\begin{enumerate}
\item[(1)] The Archimedean test function $\pi\mapsto f(\pi_\infty)$ is non-negative and depends on the following datas (c.f. \eqref{Definitionfinfty} in Section \ref{SectionInterlude}) :
\begin{itemize}
\item[$\bullet$] The choice of test vectors $\varphi_{i,\infty}\in \pi_{i,\infty}$ for $i=1,2$;
\item[$\bullet$] The choice of an orthonormal basis $\Bscr(\pi_\infty)$.
\end{itemize}
\item[(2)] Under the identification $\pi=\otimes_v \pi_v$, the correction factor $\Lcal(\pi,\qmf)$ factorizes as $\prod_{v|\qmf}\Lcal(\pi_v,\qmf_v)$ with $\Lcal(\pi_v,\qmf_v)=1$ if $\crm(\pi_v)=1$ and $0<\Lcal(\pi_v,\qmf_v)\ll 1$ if $\crm(\pi_v)=0$ (c.f. \eqref{Lcal} and \eqref{ellv}), where the natural number $\crm(\pi_v)$ is the conductor of the local representation $\pi_v$, in the sense of \cite[Theorem 4.24]{adele}. 
\end{enumerate}
For the continuous part, we denote by $\pi_{\omega}(it)$ the principal series $\omega|\cdot|^{it}\boxplus \omegabar|\cdot|^{-it}$ and define similarly
\begin{equation}\label{ContinuousPart}
\begin{split}
\mathscr{E}(\pi_1,\pi_2,\qmf,\lmf):= C\sum_{\substack{\omega\in\widehat{\F^\times\setminus \Abb_\F^{1}} \\ \crm(\omega)=\Ocal_\F}}&\int_{-\infty}^\infty \lambda_{\pi_\omega(it)}(\lmf)f(\pi_{\omega_\infty}(it))\Lcal(\pi_\omega(it),\qmf) \\ \times & \frac{L(\pi_1\otimes\pi_2\otimes\omega,\tfrac{1}{2}+it)L(\pi_1\otimes\pi_2\otimes\omegabar,\tfrac{1}{2}-it)}{\Lambda^*(\pi_\omega(it),\mathrm{Ad},1)} \frac{dt}{4\pi},
\end{split}
\end{equation}
and
\begin{equation}\label{DefinitionMoment1}
\Mscr(\pi_1,\pi_2,\qmf,\lmf) =\mathscr{C}(\pi_1,\pi_2,\qmf,\lmf)+ \mathscr{E}(\pi_1,\pi_2,\qmf,\lmf).
\end{equation}
The first theorem establishes an upper bound for this twisted moment.
%%%%%%%%%%%%%%%%%%%%%%%%%%%%%%%%%%%%%%%%%%%%%%%%%%%%%%
%%%%%%%%%%%%%%%%%%%%%%%%%%%%%%%%%%%%%%%%%%%%%%%%%%%%%%
%%%%%								THEOREM 1
%%%%%%%%%%%%%%%%%%%%%%%%%%%%%%%%%%%%%%%%%%%%%%%%%%%%%%
%%%%%%%%%%%%%%%%%%%%%%%%%%%%%%%%%%%%%%%%%%%%%%%%%%%%%%
\begin{mainthm}\label{Theorem1}
Let $\pi_1,\pi_2$ be unitary automorphic representations of $\PGL_2(\Abb_\F)$ satisfying Hypothesis \ref{Hyp1}. Let $\qmf,\lmf$ be two coprime ideals of $\Ocal_\F$ with $\qmf$ squarefree and write $q$ and $\ell$ for their respective norms. Assuming that $q>\ell$, the moment \eqref{DefinitionMoment1} satisfies 
$$\Mscr(\pi_1,\pi_2,\qmf,\lmf) \ll_{\pi_1,\pi_2,\F,\varepsilon} (\ell q)^\varepsilon \left(\ell^{1/2}q^{1/2+\vartheta}+q\ell^{-1/2}\right).$$
\end{mainthm}
%%%%%%%%%%%%%%%%%%%%%%%%%%%%%%%%%%%%%%%%%%%%%%%%%%%%%%
In this paper, the real number $\vartheta$ denotes an admissible exponent toward the Ramanujan-Petersson Conjecture for $\GL_2$ over $\F$; we have $0\leqslant \vartheta\leqslant \frac{7}{64}$ by \cite{ramanujan}. Combining Theorem \ref{Theorem1} with the amplification method, we obtain the following subconvexity bound for the triple product :
%%%%%%%%%%%%%%%%%%%%%%%%%%%%%%%%%%%%%%%%%%%%%%%%%%%%%%
%%%%%%%%%%%%%%%%%%%%%%%%%%%%%%%%%%%%%%%%%%%%%%%%%%%%%%
%%%%%									THEOREM 2
%%%%%%%%%%%%%%%%%%%%%%%%%%%%%%%%%%%%%%%%%%%%%%%%%%%%%%
%%%%%%%%%%%%%%%%%%%%%%%%%%%%%%%%%%%%%%%%%%%%%%%%%%%%%%
\begin{mainthm}\label{Theorem2}
Let $\F$ be a number field with ring of integers $\Ocal_\F$. Let $\qmf$ be a squarefree ideal of $\Ocal_\F$ of norm $q$ and $\pi$ a cuspidal automorphic representation of $\PGL_2(\Abb_\F)$ with finite conductor $\qmf$. Let $\pi_1,\pi_2$ be as in Theorem \ref{Theorem1} and assume moreover that they satisfy Hypothesis \ref{Hyp}. Then for any $\varepsilon>0$, we have the following subconvex estimate
\begin{equation}\label{SubConv1}
L\left( \pi\otimes\pi_1\otimes\pi_2,\tfrac{1}{2}\right) \ll_{\varepsilon,\F,\pi_1,\pi_2,\pi_\infty} q^{1-\frac{1-2\vartheta}{6}+\varepsilon}.
\end{equation}
In particular, choosing $\pi_2=|\cdot|^{it}\boxplus |\cdot|^{-it}$ in $(a)-(b)$ and $\pi_1=1\boxplus 1$ in $(b)$ below, Hypothesis \ref{Hyp} is automatically satisfied and we obtain
\begin{enumerate}
\item[$(a)$] For $t\in\Rbb$ and $\pi_1$ cuspidal,
\begin{equation}\label{SubConv2}
L\left(\pi\otimes \pi_1,\tfrac{1}{2}+it\right) \ll_{\varepsilon,\F,\pi_1,\pi_\infty,t} q^{\frac{1}{2}-\frac{1-2\vartheta}{12}+\varepsilon}.
\end{equation}
\item[$(b)$] For any $t\in \Rbb$,
\begin{equation}\label{SubConv3}
L\left(\pi,\tfrac{1}{2}+it\right)\ll_{\varepsilon,\F,\pi_\infty,t}q^{\frac{1}{4}-\frac{1-2\vartheta}{24}+\varepsilon}.
\end{equation}
\end{enumerate}
\end{mainthm}
\begin{remq}We emphasizes that Theorems \ref{Theorem1} and \ref{Theorem2} contain the same numerology of exponents as \cite{blomerspectral}, but the factorization of the degree $8$ $L$-function is different. It coincides in the special case of the fourth moment, i.e. when $\pi_1=\pi_2=1\boxplus 1$.
\end{remq}
%%%%%%%%%%%%%%%%%%%%%%%%%%%%%%%%%%%%%%%%%%%%%%%%%%%%%%%
%%%%%%%%%%%%%%%%%%%%%%%%%%%%%%%%%%%%%%%%%%%%%%%%%%%%%%%
%%%%%				REMARK THEOREM 2
%%%%%%%%%%%%%%%%%%%%%%%%%%%%%%%%%%%%%%%%%%%%%%%%%%%%%%%
%%%%%%%%%%%%%%%%%%%%%%%%%%%%%%%%%%%%%%%%%%%%%%%%%%%%%%%
\begin{remq} We would like to mention some earlier results concerning the subconvexity problem in the level aspect for $L(\pi\otimes\pi_1\otimes\pi_2,\tfrac{1}{2})$.
\begin{enumerate}[(1)]
\item 
In the case where both $\pi_1$ and $\pi_2$ are two fixed cuspidal representations with finite conductor $\textfrak{n}_1$ and $\textfrak{n}_2$, Venkatesh’s work \cite{sparse} together with Woodbury’s work on local integrals in \cite{wood} can be used to derive the subconvexity bound 
\begin{equation}\label{1/12}
L(\pi\otimes\pi_1\otimes\pi_2,\tfrac{1}{2})\ll q^{1-\tfrac{1}{12}},
\end{equation}
for $\pi$ cuspidal with squarefree finite conductor $\qmf$ which is prime with $\textfrak{n}_1\textfrak{n}_2$. This result is based on the condition that a certain local epsilon factor $\varepsilon_v$ associated to $\pi_v\otimes\pi_{1,v}\otimes \pi_{2,v}$ is equal to $1$ for every place $v$. In \cite{hu}, Hu removed these local conditions and extended the bound \eqref{1/12} to arbitray level $\qmf$ not necessarily prime with $\textfrak{n}_1\textfrak{n}_2$, and with arbitrary central characters. Our subconvex bound \eqref{SubConv1} improves therefore on the exponent $\tfrac{1}{12}$, but conditionally on Hypothesis \ref{Hyp1} and \ref{Hyp}.
\item Even over the field of rational numbers $\Qbb$, the exponent \eqref{SubConv2} is new, but it works if $\pi_1$ is tempered at the finite places and both $\pi,\pi_1$ have trivial central character, which is for example the case in \cite{michel0} since the fixed representation corresponds to an holomorphic cusp form; the authors obtained the exponent $\tfrac{1}{2}-\tfrac{1}{80}$. For $\pi_1$ holomorphic, we can cite also \cite{michel1,holo} where this time the product of the central characters $\omega_{\pi}\omega_{\pi_1}$ is not trivial. The holomorphicity for $\pi_1$ has been removed in \cite{michel2} with an exponent of $\tfrac{1}{2}-\tfrac{1}{2648}$. Finally, for a general number field $\F$, we mention the work of Michel and Venkatesh \cite{sparse,subconvexity} where the problem was solved in great generality.
\item The bound \eqref{SubConv3} is an extension of \cite[Theorem 4]{blomerspectral} to an arbitrary number field, the previous record being due to H. Wu with $\tfrac{1}{4}-\tfrac{1-2\vartheta}{32}$ \cite[Corollary 1.6]{explicit}. However, when the ground field is $\F=\Qbb$, the exponent has been improved to $\tfrac{1}{5}+\tfrac{\vartheta}{15}$ \cite[Corollary 3]{moto}.
\end{enumerate}
\end{remq}
\begin{remq} The bound in Theorem \ref{Theorem1} depends in fact also on the choice of test vectors at the Archimedean places $\varphi_{i,v}\in\pi_{i,v}$, $i=1,2$ that we use to connect the moment \eqref{DefinitionMoment1} to a global period of the form \eqref{Type1}. However, we can fix them by choosing for each $v|\infty$ the minimal weight vector $\varphi_{i,v}\in\pi_{i,v}$. This dependence no longer appears in Theorem \ref{Theorem2} because here the $\varphi_{i,v}$'s are completely determined by $\pi_\infty$ for a given $\pi$ (c.f. Section \ref{SectionInterlude} and \ref{SectionSub}).
\end{remq}

%%%%%%%%%%%%%%%%%%%%%%%%%%%%%%%%%%%%%%%%%%%%%
%%%%%%%%%%%%%%%%%%%%%%%%%%%%%%%%%%%%%%%%%%%%%
%%%%%%%%%%%%%%%%%%%%%%%%%%%%%%%%%%%%%%%%%%%%%
%%%%%%%%%%%%%%%%%%%%%%%%%%%%%%%%%%%%%%%%%%%%%
%%%%%				SUBSUBSCTION RECIPROCITY FORMULAS
%%%%%%%%%%%%%%%%%%%%%%%%%%%%%%%%%%%%%%%%%%%%%
%%%%%%%%%%%%%%%%%%%%%%%%%%%%%%%%%%%%%%%%%%%%%
%%%%%%%%%%%%%%%%%%%%%%%%%%%%%%%%%%%%%%%%%%%%%
%%%%%%%%%%%%%%%%%%%%%%%%%%%%%%%%%%%%%%%%%%%%%
\subsubsection{Reciprocity and simultaneous non-vanishing for the symmetric square} Assume that $\pi_1=\pi_2=\sigma$ is either cuspidal or the standard Eisenstein series $1\boxplus 1$. As announced at the end of Section \ref{Section1}, we can obtain a reciprocity relation in this case (up to a main term). To do this, we make the additional assumption that $\lmf$ is squarefree (of course coprime with $\qmf$) and we set $\Mscr(\sigma,\sigma,\qmf,\lmf)=: \Mscr(\sigma,\qmf,\lmf)$ (recall Definition \eqref{DefinitionMoment1}) and
$$\widetilde{\mathscr{M}}(\sigma,\lmf,\qmf) := \widetilde{\Cscr}(\sigma,\lmf,\qmf)+\widetilde{\mathscr{E}}(\sigma,\lmf,\qmf),$$
where $\widetilde{\Mscr}$ is the same as $\Mscr$, but with $f(\pi_\infty)$ and $\Lcal(\pi,\lmf)$ replaced by other weight functions $\widetilde{f}$ and $\widetilde{\Lcal}$ (c.f. Section \ref{SectionProof34}). Observe that in this special case, the triple product $L$-function admits the following factorization \cite[eq. (1.5)]{nelson}
$$
L(\pi\otimes \sigma\otimes\sigma,\tfrac{1}{2})= \left\{ \begin{array}{lcl}
L(\pi,\tfrac{1}{2})^4 & \ifm & \sigma=1\boxplus 1 \\ 
 & & \\ 
 L(\mathrm{Sym}^2(\sigma)\otimes \pi,\tfrac{1}{2})L(\pi,\tfrac{1}{2}) & \ifm & \sigma \ \mathrm{is \ cuspidal}.
\end{array}
\right.
$$
%%%%%%%%%%%%%%%%%%%%%%%%%%%%%%%%%%%%%%%%%%%%%
%%%%%%%%%%%%%%%%%%%%%%%%%%%%%%%%%%%%%%%%%%%%%
%%%%%%%%%%%%%%%%%%%%%%%%%%%%%%%%%%%%%%%%%%%%%
%%%%%				THEOREM 3
%%%%%%%%%%%%%%%%%%%%%%%%%%%%%%%%%%%%%%%%%%%%%
%%%%%%%%%%%%%%%%%%%%%%%%%%%%%%%%%%%%%%%%%%%%%
%%%%%%%%%%%%%%%%%%%%%%%%%%%%%%%%%%%%%%%%%%%%%
\begin{mainthm}\label{Theorem3} Let $\sigma$ be a unitary automorphic representation of $\PGL_2(\Abb_\F)$ which is unramified at the finite places and let $\qmf,\lmf$ two squarefree and coprime ideals with respective norm $q$ and $\ell$. Assume that $\ell<q$; then we have the reciprocity relation
$$\Mscr(\sigma,\qmf,\lmf)=\left( \frac{q}{\ell}\right)^{1/2}\widetilde{\Mscr}(\sigma,\lmf,\qmf)+  \mathrm{MT}(\sigma,\qmf,\lmf),$$
where the main term $\mathrm{MT}(\sigma,\qmf,\pmf)$ satisfies
$$\mathrm{MT}(\sigma,\qmf,\lmf)=\left\{ \begin{array}{lcl} \mathscr{S}(\sigma,\qmf,\lmf)+O_{\sigma,\F,\varepsilon}\left(q^{2\vartheta_{\sigma}+\varepsilon} \ell^{1/2}\right)& \mathrm{if} & \sigma \ \mathrm{is \ cuspidal} \\ & & \\O_{\sigma,\F,\varepsilon}\left(q^{1+\varepsilon}\ell^{-1/2} \right) & \mathrm{if} & \sigma \ \mathrm{is \ Eisenstein}, \end{array}\right.$$
with
$$\mathscr{S}(\sigma,\qmf,\lmf) :=\frac{4q\Delta_\F\Lambda(\sigma,\mathrm{Ad},1)^2\zeta_{\qmf}(1)}{V_\F \zeta_{\qmf}(2)}\times  \frac{\lambda_{\sigma}(\lmf)^2\zeta_{\lmf}(2)}{\ell^{1/2}\zeta_{\lmf}(1)}.$$
Here $V_\F$ and $\Delta_\F$ are constant associated to $\F$ defined in Section $\ref{SectionNotations}$, $\Lambda(\sigma,\mathrm{Ad},1)$ is the complete adjoint $L$-function of $\sigma$ and for any ideal $\amf$ and complex $s$, the partial zeta function $\zeta_\amf(s)$ is given by $\prod_{v|\amf}\zeta_{\F_v}(s)$.
\end{mainthm} 
As a consequence of Theorem \ref{Theorem3}, we obtain the following asymptotic expansion and a simultaneous non-vanishing result (see also \cite{simult}).

%%%%%%%%%%%%%%%%%%%%%%%%%%%%%%%%%%%%%%%%%%%%%
%%%%%%%%%%%%%%%%%%%%%%%%%%%%%%%%%%%%%%%%%%%%%
%%%%%%%%%%%%%%%%%%%%%%%%%%%%%%%%%%%%%%%%%%%%%
%%%%%					THEOREM 4
%%%%%%%%%%%%%%%%%%%%%%%%%%%%%%%%%%%%%%%%%%%%%
%%%%%%%%%%%%%%%%%%%%%%%%%%%%%%%%%%%%%%%%%%%%%
%%%%%%%%%%%%%%%%%%%%%%%%%%%%%%%%%%%%%%%%%%%%%
\begin{mainthm}\label{Theorem4} Let $\sigma$ be a cuspidal automorphic representation of $\PGL_2(\Abb_\F)$ unramified at the finite places. Let $\qmf$ be a prime ideal of $\Ocal_\F$ of norm $q$. Then for any $\varepsilon>0$, we have the asymptotic formula
$$\frac{1}{q}\sum_{\substack{\pi \ \mathrm{cuspidal} \\ \crm(\pi)=\qmf}} \frac{L\left(\mathrm{Sym}^2(\sigma)\otimes \pi,\tfrac{1}{2}\right)L(\pi,\tfrac{1}{2})}{\Lambda(\pi,\mathrm{Ad},1)}f(\pi_\infty)= \frac{2 \Delta_\F \Lambda(\sigma,\mathrm{Ad},1)^2}{\Lambda_\F(2)V_\F}+O_{\sigma,\F,\varepsilon}\left(q^{-\tfrac{1}{2}+\vartheta+\varepsilon}\right).$$
\end{mainthm}

%%%%%%%%%%%%%%%%%%%%%%%%%%%%%%%%%%%%%%%%%%%%%
%%%%%%%%%%%%%%%%%%%%%%%%%%%%%%%%%%%%%%%%%%%%%
%%%%%%%%%%%%%%%%%%%%%%%%%%%%%%%%%%%%%%%%%%%%%
%%%%%				COROLLARY 5
%%%%%%%%%%%%%%%%%%%%%%%%%%%%%%%%%%%%%%%%%%%%%
%%%%%%%%%%%%%%%%%%%%%%%%%%%%%%%%%%%%%%%%%%%%%
%%%%%%%%%%%%%%%%%%%%%%%%%%%%%%%%%%%%%%%%%%%%%
\begin{cor}\label{Corollary5}
Let $\sigma$ cuspidal satisfying Hypothesis \ref{Hyp1} and \ref{Hyp}, $\qmf$ a prime ideal of norm $q$ and let $\varepsilon>0$. There exists a positive constant $\lambda=\lambda(\sigma,\F,\varepsilon)$ such that for $q$ large enough in terms of $\sigma$ and $\varepsilon$, we have
$$\left| \left\{ \begin{matrix}  \pi \ \mathrm{cuspidal} \\ \crm(\pi)=\qmf \end{matrix} \ : \ \begin{matrix} L\left(\mathrm{Sym}^2(\sigma)\otimes \pi,\tfrac{1}{2}\right)\neq 0 \\ L(\pi,\tfrac{1}{2})\neq 0 \end{matrix}  \right\} \right| \geqslant \lambda q^{\frac{1-2\vartheta}{6}-\varepsilon}.$$
\end{cor}

%%%%%%%%%%%%%%%%%%%%%%%%%%%%%%%%%%%%%%%%%%%%%
%%%%%%%%%%%%%%%%%%%%%%%%%%%%%%%%%%%%%%%%%%%%%
%%%%%%%%%%%%%%%%%%%%%%%%%%%%%%%%%%%%%%%%%%%%%
%%%%%%%%%%%%%%%%%%%%%%%%%%%%%%%%%%%%%%%%%%%%%
%%%%%				SECTION : REGULARIZED PARSEVAL FORMULA
%%%%%%%%%%%%%%%%%%%%%%%%%%%%%%%%%%%%%%%%%%%%%
%%%%%%%%%%%%%%%%%%%%%%%%%%%%%%%%%%%%%%%%%%%%%
%%%%%%%%%%%%%%%%%%%%%%%%%%%%%%%%%%%%%%%%%%%%%
%%%%%%%%%%%%%%%%%%%%%%%%%%%%%%%%%%%%%%%%%%%%%
\section{Automorphic Preliminaries}

\subsection{Notations and conventions}\label{SectionNotations}

\subsubsection{Number fields} In this paper, $\F/\mathbb{Q}$ will denote a fixed number field with ring of intergers $\Ocal_\F$ and discriminant $\Delta_\F$. We make the assumption that all prime ideals considering in this paper do not divide $\Delta_\F$. We let $\Lambda_\F$ be the complete $\zeta$-function of $\F$; it has a simple pole at $s=1$ with residue $\Lambda_\F^*(1)$.

%%%%%%%%%%%%%%%%%%%%%%%%%%%%%%%%%%%%%%%%%%%%%%%%%%%%%%%%%%%%%%%%%%%%%%%%%%%
%%%%%%%%%%%%%%%%%%%%%%%%%%%%%%%%%%%%%%%%%%%%%%%%%%%%%%%%%%%%%%%%%%%%%%%%%%%
\subsubsection{Local fields} For $v$ a place of $\F$, we set $\F_v$ for the completion of $\F$ at the place $v$. We will also write $\F_\pmf$ if $v$ is finite place that corresponds to a prime ideal $\pmf$ of $\Ocal_\F$. If $v$ is non-Archimedean, we write $\Ocal_{\F_v}$ for the ring of integers in $\F_v$ with maximal $\mmf_v$ and uniformizer $\varpi_v$. The size of the residue field is $q_v=\Ocal_{\F_v}/\mmf_v$. For $s\in\Cbb$, we define the local zeta function $\zeta_{\F_v}(s)$ to be $(1-q_v^{-s})^{-1}$ if $v<\infty$, $\zeta_{\F_v}(s)=\pi^{-s/2}\Gamma(s/2)$ if $v$ is real and $\zeta_{\F_v}(s)=2(2\pi)^{-s}\Gamma(s)$ if $v$ is complex.

%%%%%%%%%%%%%%%%%%%%%%%%%%%%%%%%%%%%%%%%%%%%%%%%%%%%%%%%%%%%%%%%%%%%%%%%%%%
%%%%%%%%%%%%%%%%%%%%%%%%%%%%%%%%%%%%%%%%%%%%%%%%%%%%%%%%%%%%%%%%%%%%%%%%%%%
\subsubsection{Adele ring} The adele ring of $\F$ is denoted by $\Abb_\F$ and its unit group $\Abb^\times_\F$. We also set $\widehat{\Ocal}_\F=\prod_{v<\infty}\Ocal_{\F_v}$ for the profinite completion of $\Ocal_\F$ and $\Abb^1_\F=\{ x\in \Abb_\F^\times \ : \ |x|=1\}$, where $|\cdot| : \Abb_\F^\times \rightarrow \Rbb_{>0}$ is the adelic norm map.

%%%%%%%%%%%%%%%%%%%%%%%%%%%%%%%%%%%%%%%%%%%%%%%%%%%%%%%%%%%%%%%%%%%%%%%%%%%
%%%%%%%%%%%%%%%%%%%%%%%%%%%%%%%%%%%%%%%%%%%%%%%%%%%%%%%%%%%%%%%%%%%%%%%%%%%
\subsubsection{Additive characters} \label{Discri}We denote by $\psi = \prod_v \psi_v$ the additive character $\psi_\Qbb\circ \Tr_{\F/\Qbb}$ where $\psi_\Qbb$ is the additive character on $\Qbb\setminus \Abb_\Qbb$ with value $e^{2\pi i x}$ on $\Rbb$. For $v<\infty$, we let $d_v$ be the conductor of $\psi_v$ : this is the smallest non-negative integer such that $\psi_v$ is trivial on $\mmf_v^{d_v}$. We have in this case $\Delta_\F=\prod_{v<\infty}q_v^{d_v}$. We also set $d_v=0$ for $v$ Archimedean.

%%%%%%%%%%%%%%%%%%%%%%%%%%%%%%%%%%%%%%%%%%%%%%%%%%%%%%%%%%%%%%%%%%%%%%%%%%%
%%%%%%%%%%%%%%%%%%%%%%%%%%%%%%%%%%%%%%%%%%%%%%%%%%%%%%%%%%%%%%%%%%%%%%%%%%%
\subsubsection{Subgroups}\label{Sectionsubgroups} If $R$ is a commutative ring, $\GL_2(R)$ is by definition the group of $2\times 2$ matrices with coefficients in $R$ and determinant in $R^*$. We also defined the following standard subgroups
$$\Brm(R)=\left\{\left(\begin{matrix} a & b \\ & d\end{matrix}\right) \ : \ a,d\in R^*, b\in R\right\}, \ \Prm(R)= \left\{\left(\begin{matrix} a & b \\ & 1\end{matrix}\right) \ : \ a\in R^*, b\in R\right\}, \ $$
$$\Zrm(R)=\left\{\left(\begin{matrix} z & \\ & z\end{matrix}\right) \ : \ z\in R^*\right\}, \ \Abf(R)=\left\{\left(\begin{matrix} a &  \\ & 1\end{matrix}\right) \ : \ a\in R^*\right\}, $$ $$\Nm(R)=\left\{\left(\begin{matrix} 1 & b \\ & 1\end{matrix}\right) \ : \ b\in R\right\}.$$
We also set
$$n(x)=\left( \begin{matrix}
1 & x \\ & 1
\end{matrix}\right), \hspace{0.4cm}w = \left(\begin{matrix} & 1 \\ -1 & \end{matrix}\right) \hspace{0.4cm} \mathrm{and} \hspace{0.4cm} a(y)=\left( \begin{matrix}
y & \\ & 1
\end{matrix}\right).
$$
For any place $v$, we let $\Krm_v$ be the maximal compact subgroup of $\Grm(\F_v)$ defined by
\begin{equation}\label{Compact}
\Krm_v= \left\{ \begin{array}{lcl}
\GL_2(\Ocal_{\F_v}) & \ifm & v \ \mathrm{is \ finite} \\
 & & \\
\mathrm{O}_2(\Rbb) & \ifm & v \ \mathrm{is \ real} \\
 & & \\
\mathrm{U}_2(\mathbb{C}) & \ifm & v \ \mathrm{is \ complex}.
\end{array}\right.
\end{equation}
We also set $\Krm:= \prod_v \Krm_v$. If $v<\infty$ and $n\geqslant 0$, we define 
$$\Krm_{v,0}(\varpi_v^n)= \left\{ \left( \begin{matrix}
a & b \\ c & d
\end{matrix}\right)\in \Krm_v \ : \  c\in \mmf_v^n\right\}.$$
If $\amf$ is an ideal of $\Ocal_\F$ with prime decomposition $\amf=\prod_{v<\infty}\pmf_v^{f_v(\amf)}$ ($\pmf_v$ is the prime ideal corresponding to the finite place $v$), then we set
$$\Krm_0(\amf)=\prod_{v<\infty} \Krm_{v,0}(\varpi_v^{f_v(\amf)}).$$

%%%%%%%%%%%%%%%%%%%%%%%%%%%%%%%%%%%%%%%%%%%%%%%%%%%%%%%%%%%%%%%%%%%%%%%%%%%
%%%%%%%%%%%%%%%%%%%%%%%%%%%%%%%%%%%%%%%%%%%%%%%%%%%%%%%%%%%%%%%%%%%%%%%%%%%
\subsubsection{Measures}\label{SectionMeasure} We use the same measures normalizations as in \cite[Sections 2.1,3.1]{subconvexity}. At each place $v$, $dx_v$ denotes a self-dual measure on $\F_v$ with respect to $\psi_v$. If $v<\infty$, $dx_v$ gives the measure $q_v^{-d_v/2}$ to $\Ocal_{\F_v}$. We define $dx=\prod_v dx_v$ on $\Abb_\F$. We take $d^\times x_v=\zeta_{\F_v}(1)\frac{dx_v}{|x_v|}$ as the Haar measure on the multiplicative group $\F_v^\times$ and $d^\times x = \prod_v d^\times x_v$ as the Haar measure on the idele group $\Abb^\times_\F$.
We provide $\Krm_v$ with the probability Haar measure $dk_v$. We identify the subgroups $\Zrm(\F_v)$, $\Nrm(\F_v)$ and $\Abf(\F_v)$ with respectively $\F_v^\times,$ $\F_v$ and $\F_v^\times$ and equipped them with the measure $d^\times z$, $dx_v$ and $d^\times y_v$. Using the Iwasawa decomposition, namely $\GL_2(\F_v)=\Zrm(\F_v)\Nm(\F_v)\Abf(\F_v)\Krm_v$, a Haar measure on $\GL_2(\F_v)$ is given by 
\begin{equation}\label{HaarMeasure}
dg_v = d^\times z dx_v \frac{d^\times y_v}{|y_v|}dk_v.
\end{equation} 
The measure on the adelic points of the various subgroups are just the product of the local measures defined above. We also denote by $dg$ the quotient measure on $$\X:= \Zrm(\Abb_\F)\GL_2(\F)\setminus \GL_2(\Abb_\F),$$ 
with total mass $V_\F:=\mathrm{vol}(\X)<\infty$.

%%%%%%%%%%%%%%%%%%%%%%%%%%%%%%%%%%%%%%%%%%%%%%%%%%%%%%%%%%%%%%%%%%%%%%%%%%%
%%%%%%%%%%%%%%%%%%%%%%%%%%%%%%%%%%%%%%%%%%%%%%%%%%%%%%%%%%%%%%%%%%%%%%%%%%%
%%%%%%%%%%%%%%%%%%%%%%%%%%%%%%%%%%%%%%%%%%%%%%%%%%%%%%%%%%%%%%%%%%%%%%%%%%%
%%%%%%%%%%%			INVARIANT INNER PRODUCT ON THE WHITTAKER MODEL
%%%%%%%%%%%%%%%%%%%%%%%%%%%%%%%%%%%%%%%%%%%%%%%%%%%%%%%%%%%%%%%%%%%%%%%%%%%
%%%%%%%%%%%%%%%%%%%%%%%%%%%%%%%%%%%%%%%%%%%%%%%%%%%%%%%%%%%%%%%%%%%%%%%%%%%
%%%%%%%%%%%%%%%%%%%%%%%%%%%%%%%%%%%%%%%%%%%%%%%%%%%%%%%%%%%%%%%%%%%%%%%%%%%
\subsection{Invariant inner product on the Whittaker model}\label{SectionInvariant} Let $\pi=\otimes_v\pi_v$ be a unitary automorphic representation of $\PGL_2(\Abb_\F)$ and fix $\psi$ a character of $\F\setminus \Abb_\F$. The intertwiner 
\begin{equation}\label{NaturalIntertwiner}
\pi\ni \varphi\longmapsto W_\varphi(g)=\int_{F\setminus \Abb_\F}\varphi(n(x)g)\psi(-x)dx,
\end{equation}
realizes a $\GL_2(\Abb_\F)$-equivariant embedding of $\pi$ into a space of functions $W : \GL_2(\Abb_\F)$ $\rightarrow \Cbb$ satisfying $W(n(x)g))=\psi(x)W(g)$. The image is called the Whittaker model of $\pi$ with respect to $\psi$ and it is denoted by $\Wcal(\pi,\psi)$. This space has a factorization $\otimes_v \Wcal(\pi_v,\psi_v)$ into local Whittaker models of $\pi_v$. A pure tensor $\otimes_v \varphi_v$ has a corresponding decomposition $\prod_v W_{\varphi_v}$ where $W_{\varphi_v}(1)=1$ and is $\Krm_v$-invariant for almost all place $v$. 

\vspace{0.1cm}

We define a normalized inner product on the space $\Wcal(\pi_v,\psi_v)$ by the rule 
\begin{equation}\label{NormalizedInnerProduct}
\vartheta_v(W_v,W_v') :=\zeta_{\F_v}(2)\frac{\int_{\F_v^\times}W_v(a(y))\overline{W}_v'(a(y))d^\times y}{\zeta_{\F_v}(1)L(\pi_v,\mathrm{Ad},1)}.
\end{equation}
This normalization has the good property that $\vartheta_v(W_v,W_v)=1$ for $\pi_v$ and $\psi_v$ unramified and $W_v(1)=1$ \cite[Proposition 2.3]{classification}. We also fix for each place $v$ an invariant inner product $\langle \cdot,\cdot\rangle_v$ on $\pi_v$ and an equivariant isometry $\pi_v \rightarrow \Wcal(\pi_v,\psi_v)$ with respect to \eqref{NormalizedInnerProduct}.

\vspace{0.1cm}

Recall that if $\pi$ is a cuspidal representation, we have a unitary structure on $\pi$ given, for any $\varphi\in\pi$, by
$$||\varphi||_{L^2}^2:= \int_{\X} |\varphi|^2.$$
 If $\pi$ belongs to the continuous spectrum and $\varphi$ is the Eisenstein series associated to a section $f : \GL_2(\Abb_\F)\rightarrow \Cbb$ in an induced representation of $\Brm(\Abb_\F)$ (see for example \cite[Section 4.1.6]{subconvexity} for the basic facts and notations concerning Eisenstein series), we can define the norm of $\varphi$ by setting 
$$||\varphi||^2_{\mathrm{Eis}}:= \int_{\Krm}|f(k)|^2dk.$$
We define the canonical norm of $\varphi$ by
\begin{equation}\label{CanonicalNorm}
||\varphi||^2_{\mathrm{can}} := \left\{ \begin{array}{lcl}
||\varphi||_{L^2(\Xbf)}^2 & \ifm & \pi \ \mathrm{is \ cuspidal} \\
 & & \\
2\Lambda_\F^*(1) ||\varphi||_{\mathrm{Eis}}^2 & \ifm & \pi \ \mathrm{is \ Eisenstein},
\end{array}\right.
\end{equation}
Using \cite[Lemma 2.2.3]{subconvexity}, we can compare the global and the local inner product : for $\varphi=\otimes_v \varphi_v \in \pi=\otimes_v\pi_v$ a pure tensor with $\pi$ either cuspidal or Eisenstein and non-singular, i.e. $\pi=\chi_1\boxplus\chi_2$ with $\chi_i$ unitary, $\chi_1\chi_2=1$ and $\chi_1\neq\chi_2,$ we have
\begin{equation}\label{Comparition}
||\varphi||_{\mathrm{can}}^2=2 \Delta_\F^{1/2} \Lambda^*(\pi,\mathrm{Ad},1)\prod_v \langle \varphi_{v},\varphi_v\rangle_v,
\end{equation}
where $\Lambda(\pi,\mathrm{Ad},s)$ is the complete adjoint $L$-function $\prod_v L(\pi,\mathrm{Ad},s)$ and $\Lambda^*(\pi,\mathrm{Ad},1)$ is the first nonvanishing coefficient in the Laurent expansion around $s=1$. This regularized value satisfies \cite{adjoint}
\begin{equation}\label{BoundAdjoint}
\Lambda^*(\pi,\mathrm{Ad},1)=\mathrm{C}(\pi)^{o(1)}, \ \ \mathrm{as} \ \mathrm{C}(\pi)\rightarrow \infty,
\end{equation}
where $\mathrm{C}(\pi)$ is the analytic conductor of $\pi$, as defined in \cite[Section 1.1]{subconvexity}.

%%%%%%%%%%%%%%%%%%%%%%%%%%%%%%%%%%%%%%%%%%%%%%%%%%%%%%
%%%%%%%%%%%%%%%%%%%%%%%%%%%%%%%%%%%%%%%%%%%%%%%%%%%%%%
%%%%%%%%%%%%%%%%%%%%%%%%%%%%%%%%%%%%%%%%%%%%%%%%%%%%%%
%%%%%%%%%%%%%%%%%%%%%%%%%%%%%%%%%%%%%%%%%%%%%%%%%%%%%%
%%%%%						REGULARIZED PLANCHEREL FORMULA
%%%%%%%%%%%%%%%%%%%%%%%%%%%%%%%%%%%%%%%%%%%%%%%%%%%%%%
%%%%%%%%%%%%%%%%%%%%%%%%%%%%%%%%%%%%%%%%%%%%%%%%%%%%%%
%%%%%%%%%%%%%%%%%%%%%%%%%%%%%%%%%%%%%%%%%%%%%%%%%%%%%%
%%%%%%%%%%%%%%%%%%%%%%%%%%%%%%%%%%%%%%%%%%%%%%%%%%%%%%
\subsection{Regularized Plancherel formula}
\noindent  In this section, we define a regularization process to define integrals of non-necessarily decaying functions on $\X$. Such a regularization was first introduced by Zagier \cite{zagier} and then developped adelically by Michel and Venkatesh in \cite{subconvexity} to solve the subconvexity problem for $\GL_2$ over an arbitrary number field. We mention also the work of Han Wu \cite{wu}.

%%%%%%%%%%%%%%%%%%%%%%%%%%%%%%%%%%%%%%%%%%%%
%%%%%%%%%%%%%%%%%%%%%%%%%%%%%%%%%%%%%%%%%%%%
%%%%%%%%%%%%%%%%%%%%%%%%%%%%%%%%%%%%%%%%%%%%
%%%%%		SUBSECTION : FINITELY REGULARIZABLE FUNCTIONS
%%%%%%%%%%%%%%%%%%%%%%%%%%%%%%%%%%%%%%%%%%%%
%%%%%%%%%%%%%%%%%%%%%%%%%%%%%%%%%%%%%%%%%%%%
%%%%%%%%%%%%%%%%%%%%%%%%%%%%%%%%%%%%%%%%%%%%
\subsubsection{Finitely regularizable functions}
\begin{defi}\label{Definition1} Let $\omega$ be a unitary character on $\F^\times\setminus \Abb^\times_\F$ and $\varphi$ a smooth function on $\GL_2(\F)\setminus \GL_2(\Abb_\F)$ which transform by $\omega$ under the center $\Zrm(\Abb_\F)$. We say that $\varphi$ is $\textit{finitely \ regularizable}$ if there exists characters $\chi_i : \F^\times\setminus \Abb^\times_\F\rightarrow\Cbb^{(1)}$, $\alpha_i\in\Cbb$, $n_i\in\Nbb$ and smooth functions $f_i\in\mathrm{Ind}_{\Bm(\Abb_\F)\cap\Krm}^{\Krm}(\chi_i,\omega\chi_i^{-1})$, $i=1,...,m$ such that for any $N\geqslant 1$
$$\varphi(n(x)a(y)k)=\varphi_{\mathrm{N}}^*(a(y)k)+O(|y|^{-N}), \ \ \mathrm{as} \ |y|\rightarrow \infty,$$
where the $\mathit{essential \ constant \ term}$ is defined by
$$\varphi_{\mathrm{N}}^*(n(x)a(y)k)=\varphi_{\mathrm{N}}^*(a(y)k)=\sum_{i=1^m}\chi_i(y)|y|^{1/2+\alpha_i}(\log|y|)^{n_i}f_i(k).$$
We define the set of exponents of $\varphi$ by
$$\Exp(\varphi):= \left\{ \chi_i|\cdot|^{1/2+\alpha_i} \ : \ 1\leqslant i\leqslant m\right\},$$
and $\Exp(\varphi)^+$ (resp. $\Exp(\varphi)^-$) for those exponents with $\Re e (\alpha_i)\geqslant 0$ (resp. $\Re e (\alpha_i)<0$). Finally, we write $\Acal^{\mathrm{fr}}(\GL_2,\omega)$ for the space of finitely regularizable functions.
\end{defi}

%%%%%%%%%%%%%%%%%%%%%%%%%%%%%%%%%%%%%%%%%%%%
%%%%%					REMARK 1
%%%%%%%%%%%%%%%%%%%%%%%%%%%%%%%%%%%%%%%%%%%%
\begin{remq}\label{Remark1}
The space $\Acal^{\mathrm{fr}}(\GL_2,\omega)$ contains for example functions of the form $g\mapsto \chi(\det g)$ for $\chi$ a quasi-character satisfying $\chi^2=\omega$, smooth cusp forms with central character $\omega$ and Eisenstein series. The set of exponents of an Eisenstein series $\Erm$ associated to the induced representation $\chi_1\boxplus\chi_2$ is $\left\{\chi_1 |\cdot|^{1/2},\chi_2 |\cdot|^{1/2}\right\}$. Moreover, $\Acal^{\mathrm{fr}}(\GL_2,\omega)$ contains elements of the form $\varphi_1\varphi_2$ with $\varphi_i\in \Acal^{\mathrm{fr}}(\GL_2,\omega_i)$ and $\omega_1\omega_2=\omega$. We have
$$\Ecal\mathrm{x}(\varphi_1\varphi_2)=\Ecal\mathrm{x}(\varphi_1)\Ecal\mathrm{x}(\varphi_2).$$
Of course $\Ecal\mathrm{x}(\varphi)=\emptyset$ if $\varphi$ is cuspidal.
\end{remq}

%%%%%%%%%%%%%%%%%%%%%%%%%%%%%%%%%%%%%%%%%%%%
%%%%%			DEFINITION : THE EISENSTEIN SPACE
%%%%%%%%%%%%%%%%%%%%%%%%%%%%%%%%%%%%%%%%%%%%
\begin{defi} Let $\omega$ be a unitary character of $\Fbf^\times\setminus \Abb^\times_\F$. We denote by $\Ecal(\GL_2,\omega)$ the $\Cbb$-vector space spanned by the derivatives of all Eisenstein series of central character $\omega$.
\end{defi}

%%%%%%%%%%%%%%%%%%%%%%%%%%%%%%%%%%%%%%%%%%%%
%%%%%				THE SPACES VCAL AND WCAL
%%%%%%%%%%%%%%%%%%%%%%%%%%%%%%%%%%%%%%%%%%%%
\begin{defi} Let $S$ be a finite set of exponents, in the sense of Definition \ref{Definition1} and let $\mathcal{V}_S\subset \Acal^{\mathrm{fr}}(\GL_2,1)$ be the subspace generated by functions whose exponents belong to $S$. Define 
$$\Vcal := \bigcup_{\substack{|S|<\infty \\ 
\chi\in S \Rightarrow \chi^2\neq |\cdot|^2}}\mathcal{V}_S,$$
and 
$$\Wcal := \left\{ \varphi \in \Vcal \ | \ \Re e (\chi)\neq \frac{1}{2}, \ \forall \chi \in \Ecal\mathrm{x}(\varphi)\right\}.$$
\end{defi}

%%%%%%%%%%%%%%%%%%%%%%%%%%%%%%%%%%%%%%%%%%%%
%%%%%					THEOREM : REGULARIZED INTEGRAL
%%%%%%%%%%%%%%%%%%%%%%%%%%%%%%%%%%%%%%%%%%%%
\begin{theorem}\label{TheoremRegularized} The integral on $\Lrm^1(\X)\cap\Vcal$ extends to a $\GL_2(\Abb_\F)$-invariant linear functional on $\Vcal$, denoted by $\int^\reg_{\X}$. 
\end{theorem}

\begin{proposition}\label{PropositionL2} For every $\varphi\in \Wcal$, there exists a unique function $\Ecal(\varphi)\in\Ecal(\GL_2,\omega)$ such that $\varphi-\Ecal(\varphi)\in \Lrm^2(\X,\omega).$ In particular, the map 
$$\Wcal \longrightarrow \Ecal(\GL_2,\omega), \  \varphi\mapsto \Ecal(\varphi)$$ is $\GL_2(\Abb_\F)$-equivariant.
\end{proposition}

\begin{lemme}\label{Lemma} For any $\Ecal \in \Ecal(\GL_2,1)$, we have
$$\int_{\X}^\reg \Ecal = 0.$$
Moreover, for any $\Ecal_i\in \Ecal(\GL_2,\omega_i)$ with $\omega_1\omega_2=1$ and $\Exp(\Ecal_1)\cap\Exp(\Ecal_2)=\emptyset$, we have
$$\int_{\X}^{\mathrm{reg}}\Ecal_1\overline{\Ecal}_2=0.$$
\end{lemme}
\begin{proof}
See \cite[Lemma 3.1 $\&$ Proposition 5.27]{wu}.
\end{proof}

%%%%%%%%%%%%%%%%%%%%%%%%%%%%%%%%%%%%%%%%%%%%
%%%%%%%%%%%%%%%%%%%%%%%%%%%%%%%%%%%%%%%%%%%%
%%%%%%%%%%%%%%%%%%%%%%%%%%%%%%%%%%%%%%%%%%%%
%%%%%		SUBSECTION : A REGULARIZED PARSEVAL FORMULA
%%%%%%%%%%%%%%%%%%%%%%%%%%%%%%%%%%%%%%%%%%%%
%%%%%%%%%%%%%%%%%%%%%%%%%%%%%%%%%%%%%%%%%%%%
%%%%%%%%%%%%%%%%%%%%%%%%%%%%%%%%%%%%%%%%%%%%
\subsubsection{A regularized Plancherel formula}
Given $\varphi_1,\varphi_2\in\Acal^{\mathrm{fr}}(\GL_2,\omega)$ such that $\varphi_1\overline{\varphi}_2 \in \Vcal$, we define the regularized inner product as
$$\langle \varphi_1,\varphi_2\rangle_{\reg}:= \int_{\X}^\reg \varphi_1\overline{\varphi}_2.$$
For $p\geqslant 1$ and $\varphi$ such that $|\varphi|^p\in\Vcal$, we note also by abuse of notations
\begin{equation}\label{RegularizedNorm}
||\varphi||_{L^p,\mathrm{reg}}^p:= \int_{\X}^{\mathrm{reg}}|\varphi|^p.
\end{equation}

%%%%%%%%%%%%%%%%%%%%%%%%%%%%%%%%%%%%%%%%%%%%
%%%%%					THEOREM : PARSEVAL FORMULA
%%%%%%%%%%%%%%%%%%%%%%%%%%%%%%%%%%%%%%%%%%%%
\begin{theorem}[Regularized Plancherel Formula]\label{TheoremParseval} Let $\omega$ be a unitary Hecke character and $\varphi_1,\varphi_2 \in \Wcal$ such that the product $\varphi_1\overline{\varphi}_2$ belong to $\Vcal$. We assume moreover that the exponents of $\varphi_1$ are disjoint from the exponents of $\varphi_2$. Then
\begin{alignat*}{1}
\langle \varphi_1,&\varphi_2\rangle_\reg = \sum_{\substack{\pi \\ \mathrm{cuspidal}}}\sum_{\psi\in\Bscr(\pi)}\langle\varphi_1,\psi\rangle\langle\psi,\varphi_2\rangle +V_\F^{-1}\sum_{\chi^2=\omega}\langle \varphi_1,\varphi_\chi\rangle_\reg\langle\varphi_\chi,\varphi_2\rangle_\reg \\ 
+ & \ \sum_{\substack{\chi_1\chi_2=\omega \\ \chi_i  \in \widehat{\F^\times \setminus \Abb_\F^{(1)}}}}\int_{-\infty}^\infty \sum_{\varphi_{it}\in \Bscr(\chi_1|\cdot|^{it},\chi_2|\cdot|^{-it})}\langle \varphi_1,\Erm(,\varphi_{it})\rangle_{\reg}\langle \Erm(,\varphi_{it}),\varphi_2\rangle_{\reg}\frac{dt}{4\pi} \\ 
+ & \ \langle \varphi_1,\Ecal_2\rangle_\reg + \langle \Ecal_1,\varphi_2\rangle_\reg
\end{alignat*}
where $\Ecal_i = \Ecal(\varphi_i)$.
\end{theorem}
\begin{proof}
We follow the ideas of Michel and Venkatesh in \cite[Proposition 4.3.8]{subconvexity}. Let $\Ecal\in \Ecal(\GL_2,\omega)$ with $\Re e (\chi)\neq 1/2$ for any $\chi\in \Exp(\Ecal)$. Let $\pi$ be a generic automorphic representation  and denote by $\Bscr(\pi)$ an orthonormal basis of $\pi$. Let $\Pi$ be the representation underlying $\Ecal$ : $\Ecal$ being a sum of Eisenstein series, $\Pi$ is the direct sum of the corresponding induced representations. Then the projection operator
$$\Pi\longrightarrow \pi, \ \ f\longmapsto \sum_{\psi\in\Bscr(\pi)}\langle\Erm(f),\psi\rangle_\reg \psi,$$
defines a $\GL_2(\Abb_\F)$-equivariant map between $\Pi$ and $\pi$. This comes from the fact that if $\Bscr(\pi)=\{\psi\}$, then $\{g\cdot \psi\}$ is also an ortonormal basis and  Theorem \ref{TheoremRegularized} for the $\GL_2(\Abb_\F)$-invariance. Since $\Re e (\chi)\neq 1/2$, all components of $\Pi$ have no subquotient isomorphic to a standard automorphic representation, and thus this map is zero since $\pi$ is irreducible by definition. Similarly $\langle \Ecal_1,\Ecal_2\rangle_{\mathrm{reg}}=0$ since $\Ecal\mathrm{x}(\varphi_1)\cap \Ecal\mathrm{x}(\varphi_2)=\emptyset$ (see Lemma \ref{Lemma}).

\vspace{0.1cm}

If $\pi$ is a one-dimensional representation generated by $\varphi_\chi : g\mapsto \chi(\det g)$ with $\chi^2=\omega$, then the fact that $\langle\Ecal,\varphi_\chi\rangle_\reg=0$ follows from Lemma \ref{Lemma}. Indeed, if $\Pi$ is the underlying representation of $\Ecal$, then $\Ecal\overline{\varphi}_\chi=\tilde{\Ecal} \in \Ecal(\GL_2,1)$ and its associated representation is given by the twist $\Pi\otimes\chibar.$
We conclude using the usual Plancherel formula with $\varphi_i-\Ecal_i \in \Lrm^2(\X,\omega)$.
\end{proof}

%%%%%%%%%%%%%%%%%%%%%%%%%%%%%%%%%%%%%%%%%%%%
%%%%%						REMARK ON THE LEVEL
%%%%%%%%%%%%%%%%%%%%%%%%%%%%%%%%%%%%%%%%%%%%
\begin{remq}\label{Remark1}Assume that $\varphi_1$ is an automorphic form of level $\amf$, i.e.  is invariant under the subgroup $\Krm_1(\amf)$. Then in view of Proposition \ref{PropositionL2}, $\Ecal(\varphi_1)$ has the same level and we can add level restriction in Theorem \ref{TheoremParseval} (see \cite[Proposition 2.7]{raphael}).
\end{remq}

%%%%%%%%%%%%%%%%%%%%%%%%%%%%%%%%%%%%%%%%%%%%%%%%%%%%%%%%%%%%%%%%%%%%%%%%%%%
%%%%%%%%%%%%%%%%%%%%%%%%%%%%%%%%%%%%%%%%%%%%%%%%%%%%%%%%%%%%%%%%%%%%%%%%%%%
%%%%%%%%%%%%%%%%%%%%%%%%%%%%%%%%%%%%%%%%%%%%%%%%%%%%%%%%%%%%%%%%%%%%%%%%%%%
%%%%%%%%%%		INTEGRAL REPRESENTATION TRIPLE PRODUCT			%%%%%%%%%%
%%%%%%%%%%%%%%%%%%%%%%%%%%%%%%%%%%%%%%%%%%%%%%%%%%%%%%%%%%%%%%%%%%%%%%%%%%%
%%%%%%%%%%%%%%%%%%%%%%%%%%%%%%%%%%%%%%%%%%%%%%%%%%%%%%%%%%%%%%%%%%%%%%%%%%%
%%%%%%%%%%%%%%%%%%%%%%%%%%%%%%%%%%%%%%%%%%%%%%%%%%%%%%%%%%%%%%%%%%%%%%%%%%%
\subsection{Integral representations of triple product L-functions}\label{SectionRankin} 

\noindent Let $\pi_1,\pi_2,$ and $\pi_3$ be three unitary automorphic representations of $\PGL_2(\Abb_\F)$. Consider the linear functional on $\pi_1\otimes\pi_2\otimes\pi_3$ defined by
$$I (\varphi_1\otimes\varphi_2\otimes \varphi_3):= \int_{\X}^{\mathrm{reg}} \varphi_1\varphi_2\varphi_3.$$
This period is closely related to the central value of the triple product $L$-function $L(\pi_1\otimes\pi_2\otimes\pi_2,\tfrac{1}{2})$. In order the state the result, we write $\pi_i=\otimes_v \pi_{i,v}$ and for each $v$, we consider the matrix coefficient
\begin{equation}\label{DefinitionMatrixCoefficient}
I'_v(\varphi_{1,v}\otimes\varphi_{2,v}\otimes\varphi_{3,v}) :=\int_{\PGL_2(\F_v)}\prod_{i=1}^3\langle \pi_{i,v}(g_v)\varphi_{i,v},\varphi_{i,v}\rangle_v dg_v. 
\end{equation}
It is a fact that \cite[(3.27)]{subconvexity} 
\begin{equation}\label{Fact}
\frac{I'(\varphi_{1,v}\otimes\varphi_{2,v}\otimes\varphi_{3,v})}{\prod_{i=1}^3 \langle \varphi_{i,v},\varphi_{i,v}\rangle_v}= \zeta_{\F_v}(2)^2 \frac{L(\pi_{1,v}\otimes\pi_{2,v}\otimes\pi_{3,v},\tfrac{1}{2})}{\prod_{i=1}^3 L(\pi_{i,v},\mathrm{Ad},1)},
\end{equation}
when $v$ is non-Archimedean and all vectors are unramified. It is therefore natural to consider the normalized version
\begin{equation}\label{DefinitionNormalizedMatrixCoefficient}
I_v(\varphi_{1,v}\otimes \varphi_{2,v}\otimes \varphi_{3,v}) := \zeta_{\F_v}(2)^{-2} \frac{\prod_{i=1}^3 L(\pi_{i,v},\mathrm{Ad},1)}{L(\pi_{1,v}\otimes\pi_{2,v}\otimes\pi_{3,v},\tfrac{1}{2})} I'_v (\varphi_{1,v}\otimes \varphi_{2,v},\varphi_{3,v}).
\end{equation}

The following proposition connects the global trilinear form $I$ with the central value $L(\pi_1\otimes\pi_2\otimes\pi_3,\tfrac{1}{2})$ and the local matrix coefficients $I_v$. The proof when at least one of the $\pi_i$'s is Eisenstein and at least one is cuspidal can be found in \cite[(4.21)]{subconvexity} and is a consequence of the Rankin-Selberg method. The result when all $\pi_i$ are cuspidal is due to Ichino \cite{ichino} and when all the $\pi_i$'s are Eisenstein is \cite[Lemma 4.4.3]{subconvexity}.
%%%%%%%%%%%%%%%%%%%%%%%%%%%%%%%%%%%%%%%%%%%%%%%%%%%%%%
%%%%%%%%%%%%%%%%%%%%%%%%%%%%%%%%%%%%%%%%%%%%%%%%%%%%%%
%%%%%				PROPOSITION INTEGRAL REPRESENTATION
%%%%%%%%%%%%%%%%%%%%%%%%%%%%%%%%%%%%%%%%%%%%%%%%%%%%%%
%%%%%%%%%%%%%%%%%%%%%%%%%%%%%%%%%%%%%%%%%%%%%%%%%%%%%%
\begin{proposition}\label{PropositionIntegralRepresentation} Let $\pi_1,\pi_2,\pi_3$ be unitary automorphic representations of $\PGL_2(\Abb_\F)$ such that $\pi_1$ is either cuspidal or Eisenstein and non-singular. Let $\varphi_i = \otimes_v \varphi_{i,v}\in \otimes_v \pi_{i,v}$ be pure tensors and set $\varphi :=\varphi_1\otimes\varphi_2\otimes\varphi_3$.
\begin{enumerate}[$a)$]
\item If $\pi_2$ and $\pi_3$ are cudpidal, then
$$
\frac{|I(\varphi)|^2}{\prod_{i=1}^3 ||\varphi_i||^2_{\mathrm{can}}} = \frac{C}{8\Delta_\F^{3/2}}\cdot\frac{\Lambda(\pi_1\otimes\pi_2\otimes\pi_3,\tfrac{1}{2})}{\prod_{i=1}^3 \Lambda^*(\pi_i,\mathrm{Ad},1)}\prod_v \frac{I_v(\varphi_v)}{\prod_{i=1}^3\langle \varphi_{i,v},\varphi_{i,v}\rangle_v},
$$
with $C=\Lambda_\F(2)$ if $\pi_1$ is cuspidal and $C=1$ if $\pi_1$ is Eisenstein and non-singular.
\end{enumerate}
Assume from now that for $j=2$ or/and $j=3$, $\pi_j$ is Eisenstein of the form $|\cdot|^{it}\boxplus |\cdot|^{-it}$ for some $t\in\Rbb$, then we define $\varphi_j$ as the Eisenstein series associated to the section $f_j(0)\in |\cdot|^{it}\boxplus |\cdot|^{-it}$ which, for $\Re e (s)>0$, is defined by
\begin{equation}\label{DefinitionFlat}
f_j(g,s):= |\det(g)|^{s+1/2+it}\int_{\Abb_\F^\times} \Psi_j((0,t)g)|t|^{1+2(s+it)}d^\times t ,
\end{equation}
where $\Psi_j=\prod_v \Psi_{j,v}$ with $\Psi_{j,v} = \mathbf{1}_{\Ocal_{\F_v}}^2$ for finite $v$ and $j=2,3$. 
\begin{enumerate}
\item[$b)$] If $\pi_2$ is cuspidal and $\pi_3$ is Eisenstein as above, then
$$
\frac{|I(\varphi)|^2}{\prod_{i=1}^2 ||\varphi_i||^2_{\mathrm{can}}} = \frac{1}{4\Delta_\F}\cdot \frac{\Lambda(\pi_1\otimes\pi_2\otimes\pi_3,\tfrac{1}{2})}{\prod_{i=1}^2\Lambda^*(\pi_i,\mathrm{Ad},1)}\prod_v \frac{I_v(\varphi_{1,v}\otimes\varphi_{2,v}\otimes f_{3,v})}{\prod_{i=1}^2\langle \varphi_{i,v},\varphi_{i,v}\rangle_v}.
$$
\item[$c)$] If both $\pi_2$ and $\pi_3$ are Eisenstein, then
$$\frac{|I(\varphi)|^2}{ ||\varphi_1||^2_{\mathrm{can}}} = \frac{1}{2\Delta_\F^{1/2}}\cdot \frac{\Lambda(\pi_1\otimes\pi_2\otimes\pi_3,\tfrac{1}{2})}{\Lambda^*(\pi_1,\mathrm{Ad},1)}\prod_v \frac{I_v(\varphi_{1,v}\otimes f_{2,v}\otimes f_{3,v})}{\langle \varphi_{1,v},\varphi_{1,v}\rangle_v}.$$
\end{enumerate}
\end{proposition}
%%%%%%%%%%%%%%%%%%%%%%%%%%%%%%%%%%%%%%%%%%%%%%%%%%%%%%
%%%%%%%%%%%%%%%%%%%%%%%%%%%%%%%%%%%%%%%%%%%%%%%%%%%%%%

%%%%%%%%%%%%%%%%%%%%%%%%%%%%%%%%%%%%%%%%%%%%%%%%%%%%%%
%%%%%%%%%%%%%%%%%%%%%%%%%%%%%%%%%%%%%%%%%%%%%%%%%%%%%%
%%%%%%%%%%%%%%%%%%%%%%%%%%%%%%%%%%%%%%%%%%%%%%%%%%%%%%
%%%%%%%%%%%%%%%%%%%%%%%%%%%%%%%%%%%%%%%%%%%%%%%%%%%%%%
%%%%%						HECKE OPERATORS
%%%%%%%%%%%%%%%%%%%%%%%%%%%%%%%%%%%%%%%%%%%%%%%%%%%%%%
%%%%%%%%%%%%%%%%%%%%%%%%%%%%%%%%%%%%%%%%%%%%%%%%%%%%%%
%%%%%%%%%%%%%%%%%%%%%%%%%%%%%%%%%%%%%%%%%%%%%%%%%%%%%%
%%%%%%%%%%%%%%%%%%%%%%%%%%%%%%%%%%%%%%%%%%%%%%%%%%%%%%
\subsection{Hecke operators} \noindent Let $\pmf$ be a prime ideal of $\Ocal_\F$ of norm $p$ and $n\in\mathbb{N}$. Let $\F_\pmf$ be the completion of $\F$ at the place corresponding to the prime $\pmf$ and $\varpi_\pmf$ be a uniformizer of the ring of integer $\Ocal_{\F_{\pmf}}$. Let $\mathrm{H}_{\pmf^n}$ be the double coset in $\GL_2(\F_\pmf)$
$$\mathrm{H}_{\pmf^n}:= \GL_2(\Ocal_{\F_{\pmf}}) \left(  \begin{matrix}1 &  \\  & \varpi_{\pmf^n}
\end{matrix} \right)\GL_2(\Ocal_{\F_{\pmf}}),$$
which, for $n\geqslant 1$, has measure $p^{n-1}(p+1)$ with respect to the Haar measure on $\GL_2(\F_\pmf)$ assigning mass $1$ to $\GL_2(\Ocal_{\F_\pmf})$ \cite[Section 2.8]{sparse}. We consider the compactly supported function 
$$\mu_{\pmf^n}:= \frac{1}{p^{n/2}}\sum_{0\leqslant k\leqslant \frac{n}{2}}\mathbf{1}_{\mathrm{H}_{\pmf^{n-2k}}}.$$
Given now any $f\in \mathscr{C}^\infty(\GL_2(\Abb_\F))$, the Hecke operator $\Trm_{\pmf^n}$ is given by the convolution of $f$ with $\mu_{\pmf^n}$, i.e. for any $g\in\GL_2(\Abb_\F)$,
\begin{equation}\label{ActionHecke}
(\Trm_{\pmf^n} f)(g) = (f\star \mu_{\pmf^n})(g):= \int_{\GL_2(\F_\pmf)}f(gh)\mu_{\pmf^n}(h)dh,
\end{equation} 
and the function $h\mapsto f(gh)$ has to be understood under the natural inclusion $\GL_2(\F_\pmf)\hookrightarrow \GL_2(\Abb_\F)$. This definition extends to an arbitrary integral ideal $\amf$ by multiplicativity. 

\vspace{0.1cm}

The main advantage of this abstract definition is that it simplifies a lot when we deal with $\GL_2(\Abb_\F)$-invariant functionals. Indeed, consider the natural action of $\GL_2(\Abb_\F)$ on $\Cscr^\infty(\GL_2(\Abb_\F))$ by right translation and let $\ell : \Cscr^\infty(\GL_2(\Abb_\F))\times \Cscr^\infty(\GL_2(\Abb_\F))\rightarrow \Cbb$ be a $\GL_2(\Abb_\F)$-invariant bilinear functional. Then for any $f_1,f_2$ which are right $\GL_2(\Ocal_{\F_{\pmf}})$-invariant, we have the relation
\begin{equation}\label{RelationHecke}
\ell(\Trm_{\pmf^n}f_1,f_2)= \frac{1}{p^{n/2}}\sum_{0\leqslant k\leqslant \frac{n}{2}} \gamma_{n-2k}\ell\left(\left(\begin{matrix} 1 & \\ & \varpi^{n-2k}\end{matrix}\right)\cdot f_1,f_2 \right),
\end{equation}
with

\begin{equation}\label{ValueGamma}
\gamma_r:= \left\{ \begin{array}{lcl}
1 & \ifm & r=0 \\ 
 & & \\
p^{r-1}(p+1) & \ifm & r\geqslant 1. 
\end{array}\right.
\end{equation}

%
%
%
%%%%%%%%%%%%%%%%%%%%%%%%%%%%%%%%%%%%%%%%%%%%%
%%%%%%%%%%%%%%%%%%%%%%%%%%%%%%%%%%%%%%%%%%%%%
%%%%%%%%%%%%%%%%%%%%%%%%%%%%%%%%%%%%%%%%%%%%%
%%%%%%%%%%%%%%%%%%%%%%%%%%%%%%%%%%%%%%%%%%%%%
%%%%%%		SECTION : REGULARIZED ESTIMATIONS
%%%%%%%%%%%%%%%%%%%%%%%%%%%%%%%%%%%%%%%%%%%%%
%%%%%%%%%%%%%%%%%%%%%%%%%%%%%%%%%%%%%%%%%%%%%
%%%%%%%%%%%%%%%%%%%%%%%%%%%%%%%%%%%%%%%%%%%%%
%%%%%%%%%%%%%%%%%%%%%%%%%%%%%%%%%%%%%%%%%%%%%
\section{Estimations of Regularized Periods}\label{SectionEstimation}
Let $\pi_1,\pi_2$ be unitary automorphic representations of $\GL_2(\Abb_\F)$ with trivial central character which we assume to satisfy Hypothesis \ref{Hyp1}. To simplify the argument in Section \ref{SectionDeformation}, we suppose that if $\pi_i$ is Eisenstein, then it is the standard $1\boxplus 1$ induced from two trivial characters. Let $\varphi_i=\otimes_v \varphi_{i,v}\in \pi_i=\otimes_v \pi_{i,v}$ be $\GL_2(\widehat{\Ocal}_\F)$-invariant vectors defined as follows :
\begin{enumerate}
\item[$\bullet$]  For $\pi_i$ cuspidal, fix a unitary structure $\langle \cdot,\cdot\rangle_{i,v}$ on each $\pi_{i,v}$ compatible with \eqref{NormalizedInnerProduct} as in Section \ref{SectionInvariant} and take $\varphi_{i,v}$ to have norm $1$. 
\item[$\bullet$] If $\pi_i= 1\boxplus 1$, we let $\varphi_i$ be the Eisenstein series associated to the section defined in \eqref{DefinitionFlat}. 
\end{enumerate}

Let $\lmf$ be an ideal of $\Ocal_\F$. In order to not cluther the treatment and the combinatory with the Hecke operators, we take $\lmf$ of the form $\pmf^n$ with $\pmf\in\mathrm{Spec}(\Ocal_\F)$ and $n\in\mathbb{N}$ and set $p$ for the norm of $\pmf$, so that $\ell=p^n$ is the norm of $\lmf$. For $0\leqslant r\leqslant n$, we write as usual 
$$\varphi_i^{\pmf^r} :=\left(\begin{matrix} 1 & \\ & \varpi_\pmf^r \end{matrix}\right)\cdot \varphi_i.$$ 

In Section \ref{SectionL^2}, we establish a bound for a particular L$^2$-norm of automorphic forms. In Section \ref{SectionDegenerate}, we examine some degenerate contribution appearing in the regularized Plancherel formula. Finally, we combine both results in Section \ref{SectionGeneric} to estimate a generic expansion, which can be seen as a period analogue of \cite[Theorem 6.6]{han2}.

\begin{remq}\label{RemarkInfinite} Observe that for every finite place $v$, our local vectors $\varphi_{i,v}$ are uniquely determined. Indeed there are defined explicitely by \eqref{DefinitionFlat} in the Eisenstein case and if $\pi_i$ is cuspidal, there is a unique L$^2$-normalized $\Krm_v$-invariant vector in $\pi_v$. We allow here the infinite component $\varphi_{i,\infty}$ to have a certain degree of freedom. In fact, these will be chosen for Theorem \ref{Theorem2} in Section \ref{SectionInterlude} and will depend only on $\pi_\infty$; $\pi$ being the representation for which we want to obtain subconvexity in Section \ref{SectionSub}. We make therefore the convention that all $\ll$ involved in Sections \ref{SectionEstimation} and \ref{SectionSymmetric} depend implicitly on $\varphi_{1,\infty}$ and $\varphi_{2,\infty}$. We can also fix them by choosing for each infinite place $v$, $\varphi_{i,v}\in \pi_{i,v}$ to be the vector of minimal weight. 
\end{remq}

%%%%%%%%%%%%%%%%%%%%%%%%%%%%%%%%%%%%%%%%%%%%
%%%%%%%%%%%%%%%%%%%%%%%%%%%%%%%%%%%%%%%%%%%%
%%%%%%%%%%%%%%%%%%%%%%%%%%%%%%%%%%%%%%%%%%%%
%%%%%					SUBSECTION : A NORM ESTIMATION
%%%%%%%%%%%%%%%%%%%%%%%%%%%%%%%%%%%%%%%%%%%%
%%%%%%%%%%%%%%%%%%%%%%%%%%%%%%%%%%%%%%%%%%%%
%%%%%%%%%%%%%%%%%%%%%%%%%%%%%%%%%%%%%%%%%%%%
\subsection{A $\mathrm{L}^2$-norm}\label{SectionL^2}
%%%%%%%%%%%%%%%%%%%%%%%%%%%%%%%%%%%%%%%%%%%%
%%%%%		THE KEY PROPOSITION
%%%%%%%%%%%%%%%%%%%%%%%%%%%%%%%%%%%%%%%%%%%%
\begin{proposition}\label{PropositionL2norm} For sufficiently small $\varepsilon>0$, we have the following estimation
$$\int_{\X}^{\mathrm{reg}}\left| \varphi_1\varphi_2^{\lmf}\right|^2 \ll_{\pi_1,\pi_2,\varepsilon} \ell^{\varepsilon}.$$
\end{proposition}
\begin{remq}\label{Remarkcuspidal} The above Proposition becomes of course trivial if $\pi_1$ and $\pi_2$ are cuspidal. Indeed in this case we just apply Cauchy-Schwarz to get the bound $||\varphi_1||_{L^4}^2||\varphi_2||^2_{L^4}$. We will therefore assume for the rest of Section \ref{SectionL^2} that at least one of the $\pi_i$'s is Eisenstein.
\end{remq}

%%%%%%%%%%%%%%%%%%%%%%%%%%%%%%%%%%%%%%%%%%%%
%%%%%%%%%%%%%%%%%%%%%%%%%%%%%%%%%%%%%%%%%%%%
%%%%%	SUBSUBSECTION : A DEFORMATION ARGUMENT
%%%%%%%%%%%%%%%%%%%%%%%%%%%%%%%%%%%%%%%%%%%%
%%%%%%%%%%%%%%%%%%%%%%%%%%%%%%%%%%%%%%%%%%%%
\subsubsection{A deformation argument}\label{SectionDeformation} Assume for example that $\pi_1=1\boxplus 1$ is the standard Eisenstein series induced from the trivial characters. Then even if the regularized integral is well defined, we cannot apply Plancherel formula \ref{TheoremParseval} to the regularized inner product $\langle |\varphi_1|^2,|\varphi_2^{\lmf}|^2\rangle_{\mathrm{reg}}$ because the exponent of $|\varphi_1|^2$ is $|\cdot|$ ! To remedy this situation, we can use some deformation techniques which were developped by Michel and Venkatesh in \cite[§ 5.2.7]{subconvexity}. For $s \in \Cbb$, we write $\pi_i(s)$ for the $s$-deformation; that is $\pi_i(s)=\pi_i$ if $\pi_i$ is cuspidal, and $\pi_i(s)=|\cdot|^s\boxplus |\cdot|^{-s}$ if $\pi_i$ is Eisenstein. We write $\varphi_i(s)\in \pi_i(s)$ for the corresponding $s$-deformation and define
\begin{equation}\label{DefinitionPhi}
\Phi_1(s):=\varphi_1(s)\overline{\varphi}_1\left(\frac{\sbar}{2}\right) \ \ \mathrm{and} \ \ \Phi_2(s):=\varphi_2\left(\frac{\sbar}{3}\right)\overline{\varphi}_2\left( \frac{s}{4}\right).
\end{equation}
Then 
\begin{equation}\label{PDeforme}
s\longmapsto \Pscr(\lmf;s):= \int_{\X}^{\mathrm{reg}}\Phi_1(s) \overline{\Phi}_2(s)^{\lmf}
\end{equation}
is a well-defined holomorphic function in a neighborhood of the origin and its value at $s=0$ is given by the original quantity that we want to study, i.e.
\begin{equation}\label{Originalquantity}
\Pscr(\lmf;0)=\int_{\X}^{\mathrm{reg}}\left| \varphi_1\varphi_2^{\lmf}\right|^2.
\end{equation}
The result of Proposition \ref{PropositionL2norm} is of course true if $n=0$ (recall $\lmf=\pmf^n$). We thus assume the conclusion for every $m<n$ and we want to prove it for $m=n$. Using the identity \eqref{RelationHecke}, we make the Hecke operator $\Trm_{\lmf}$ appear as follows
\begin{equation}\label{Eq1}
\Pscr(\lmf;s)= \frac{\ell^{1/2}}{\gamma_n}\int_{\X}^{\mathrm{reg}} \Trm_{\lmf}(\Phi_1(s))\overline{\Phi}_2(s)- \sum_{1\leqslant k\leqslant \frac{n}{2}}\frac{\gamma_{n-2k}}{\gamma_n}\Pscr(\pmf^{n-2k};s),
\end{equation}
where the $\gamma_i$'s are given by \eqref{ValueGamma} and the sum on the righthand side appears only if $n\geqslant 2$. Using the induction hypothesis yields
\begin{equation}
\Pscr(\lmf;0)= \frac{\ell^{1/2}}{\gamma_n}\int_{\X}^{\mathrm{reg}} \Trm_{\lmf}(\Phi_1(0))\overline{\Phi}_2(0) + O_{\varepsilon}\left(p^{\varepsilon(n-2)}\right).
\end{equation}
We want to apply Theorem \ref{TheoremParseval} to the regularized inner product in \eqref{Eq1} for $s$ in a suitable open subset. For explicitness, we take
$$\Omega=\{s\in\Cbb \ | \ 0<|s|<\delta\}$$
for some absolute $\delta>0$. For $s\in\Omega$, the functions $\Phi_1(s)$ and $\Phi_2(s)$ satisfy the hypothesis of the regularized Plancherel formula. We obtain therefore the following decomposition
\begin{equation}\label{Parseval}
\int_{\X}^{\mathrm{reg}}\Trm_{\lmf}(\Phi_1(s))\overline{\Phi}_2(s)=\mathscr{G}(\lmf;s)+\Oscr(\lmf;s)+\Dscr(\lmf;s),
\end{equation}
where $\mathscr{G}(\lmf;s)$ denotes the generic part, $\Oscr(\lmf;s)$ the one-dimensional part and the degenerate contribution is given by 
\begin{equation}\label{DefinitionDegenerate}
\Dscr(\lmf;s)=\langle \Trm_{\lmf}\Phi_1(s), \Ecal_2(s)\rangle_{\mathrm{reg}}+\langle \Trm_{\lmf}\Ecal_1(s),\Phi_2(s)\rangle_{\mathrm{reg}},
\end{equation}
with $\Ecal_i(s)=\Ecal(\Phi_i(s))$. Note that we implicitly used the fact that 
$$\Trm_{\lmf}\Ecal_i(s)=\Ecal(\Trm_{\lmf}\Phi_i(s)),$$
which comes from the $\GL_2(\Abb_\F)$-equivariance in Proposition \ref{PropositionL2}. The generic part also defines an holomorphic function on $\Omega$. It is moreover regular at $s=0$; its value being given by (observe that $\Phi_2(s)$ is invariant under the group $\GL_2(\widehat{\Ocal}_\F)$)
\begin{equation}\label{Momentq}
\begin{split}
 & \Gscr(\lmf;0)= \sum_{\substack{\pi \ \mathrm{cuspidal} \\ \crm(\pi)=\Ocal_\F}}\lambda_\pi(\lmf)\sum_{\psi\in\Bscr(\pi,\Ocal_\F)} \langle \Phi_1,\psi\rangle\langle \psi,\Phi_2\rangle \\ 
 + & \ \frac{1}{4\pi}\sum_{\substack{\chi\in\widehat{\F^\times\setminus\Abb_\F^{1}} \\ \crm(\chi)=\Ocal_\F}}\int_{-\infty}^\infty\lambda_{\chi,it}(\lmf)\sum_{\psi_{it}\in\Bscr(\chi,\chi^{-1},it,\Ocal_\F)}\langle\Phi_1,\Erm(\psi_{it})\rangle_\reg \langle \Erm(\psi_{it}),\Phi_2\rangle_{\mathrm{reg}}dt ,
\end{split}
\end{equation}
where $\Phi_i=\Phi_i(0)$ and $\Bscr(\pi,\Ocal_\F)$ denotes an orthonormal basis of $\GL_2(\widehat{\Ocal}_\F)$-invariant vectors in $\pi$ and similarly for $\Bscr(\chi,\chi^{-1},it,\Ocal_\F)$. Because the Hecke eigenvalues satisfies the upper bound $|\lambda_\pi(\lmf)|\leqslant (n+1)\ell^{\vartheta}$, we obtain using Cauchy-Schwarz inequality
\begin{equation}\label{BoundGenericTerm}
|\Gscr(\lmf;0)|\leqslant (n+1)\ell^{\vartheta}\Big( \Gscr(\Phi_1)\Gscr(\Phi_2)\Big)^{1/2},
\end{equation}
with
\begin{equation}\label{SpectralExpPhi_i}
\begin{split}
\Gscr(\Phi_i):= & \ \sum_{\substack{\pi \ \mathrm{cuspidal} \\ \crm(\pi)=\Ocal_\F}}\sum_{\psi\in \Bscr(\pi,\Ocal_\F)}\left|\langle \Phi_i,\psi\rangle \right|^2 \\
 + & \ \frac{1}{4\pi}\sum_{\substack{\chi\in\widehat{\F^\times\setminus\Abb_\F^{1}} \\ \crm(\chi)=\Ocal_\F}}\int_{-\infty}^\infty\sum_{\psi_{it}\in\Bscr(\chi,\chi^{-1},it,\Ocal_\F)}\left|\langle\Phi_i,\Erm(\psi_{it})\rangle_\reg  \right|^2dt. 
\end{split}
\end{equation}
\begin{remq}\label{Dependance}
The quantity \eqref{SpectralExpPhi_i} is finite and depends only on $\pi_1,\pi_2$ and our choices at the infinite component $\varphi_{1,\infty}$ and $\varphi_{2,\infty}$. To be more explicit, if $\pi_i$ is cuspidal, then adding the one-dimensional contribution to \eqref{SpectralExpPhi_i} (only the trivial character counts) completes the spectral expansion and thus we obtain 
$$\Gscr(\Phi_i)=||\varphi_i||_{L^4}^4-||\varphi_i||_{L^2}^4.$$
If $\pi_i=1\boxplus 1$, then Lemma \ref{LemmaDegenerateContribution} below tells us that for sufficiently small $\varepsilon>0$, we have 
$$\Gscr(\Phi_i)\leqslant ||\varphi_i||^4_{L^4,\mathrm{reg}}+2\max_{|s|=\varepsilon}\left|\langle \Ecal(\Phi_i(s)),\Phi_i(\sbar/2)\rangle_{\mathrm{reg}} \right|,$$
where the last degenerate contribution will be computed explicitly in Section \ref{SectionDegenerate}.
\end{remq}
Assuming now that $\pi_1$ is Eisenstein (c.f. Remark \ref{Remarkcuspidal}), the one-dimensional part $\Oscr(\lmf;s)$ is zero because the exponents of $\varphi_1(s)$ are disjoint from the exponents of $\varphi_1(\sbar/2)$ and thus the contribution is zero in this case (see Lemma \ref{Lemma}). 
Moreover, from the above observations and \eqref{Parseval}, we obtain :
%%%%%%%%%%%%%%%%%%%%%%%%%%%%%%%%%%%%%%%%%%%%
%%%%%			LEMMA ON DEGENERATE CONTRIBUTION
%%%%%%%%%%%%%%%%%%%%%%%%%%%%%%%%%%%%%%%%%%%%
\begin{lemme}\label{LemmaDegenerateContribution} The additional contribution $s\mapsto \Oscr(\lmf;s)+\Dscr(\lmf;s)$ defines an holomorphic function on the open disc of radius $\delta$. In particular, for every $0<\varepsilon<\delta$, we have by Cauchy's formula
$$|\Oscr(\lmf;0)+\Dscr(\lmf;0)|\leqslant \max_{|s|=\varepsilon}|\Oscr(\lmf;s)+\Dscr(\lmf;s)|=\max_{|s=\varepsilon|}|\Dscr(\lmf;s)|.$$
\end{lemme}

%%%%%%%%%%%%%%%%%%%%%%%%%%%%%%%%%%%%%%%%%%%%
%%%%%%%%%%%%%%%%%%%%%%%%%%%%%%%%%%%%%%%%%%%%
%%%%%	SUBSUBSECTION : ESTIMATION OF DEGENERATE CONTRIBUTION
%%%%%%%%%%%%%%%%%%%%%%%%%%%%%%%%%%%%%%%%%%%%
%%%%%%%%%%%%%%%%%%%%%%%%%%%%%%%%%%%%%%%%%%%%
\subsubsection{A first estimation of $\Dscr(\lmf;s)$}\label{Paragraph} Let $\varepsilon>0$ be a small real number. For $|s|=\varepsilon$, we want to give an upper bound for the degenerate term $\Dscr(\lmf;s)$ defined in \eqref{DefinitionDegenerate}.  Because the Hecke operator $\Trm_\lmf$ is self-adjoint with respect to the regularized inner product, it is enough to bound the expression
$$\langle \Trm_{\lmf}\Ecal_1(s),\Phi_2(s)\rangle_{\mathrm{reg}}.$$
We can also assume that $\pi_1$ is Eisenstein, otherwise $\Ecal_1(s)=0$. It follows by Remark \ref{Remark1}  that the exponents of $\Phi_1(s)=\varphi_1(s)\overline{\varphi}_1(\sbar/2)$ are 
$$\left\{|\cdot|^{1+3s/2},|\cdot|^{1+s/2},|\cdot|^{1-s/2},|\cdot|^{1-3s/2}\right\},$$ and thus $\Ecal_1(s)$ is a sum of four Eisenstein series, each of them obtained from one of the following induced representations
$$|\cdot|^{1/2\pm 3s/2}\boxplus |\cdot|^{-1/2\mp 3s/2} \ \ \mathrm{and} \ \ |\cdot|^{1/2\pm s/2}\boxplus |\cdot|^{-1/2\mp s/2}.$$
Therefore, each of these Eisenstein series is an eigenfunction of the Hecke operator with eigenvalue of size at most $(n+1)\ell^{1/2+3\varepsilon/2}$ and thus
$$|\langle \Trm_{\lmf}\Ecal_1(s),\Phi_2(s)\rangle_{\mathrm{reg}}|\leqslant 4(n+1)\ell^{\frac{1+3\varepsilon}{2}}|\langle\Ecal_1(s),\Phi_2(s)\rangle_{\mathrm{reg}}|.$$
Of course the last contribution depends only on $\varepsilon$, $\pi_1,\pi_2$ and the infinite datas, which is satisfactory for Proposition \ref{PropositionL2norm}. However, as said in Remark \ref{Dependance}, it will be evaluated in details in the next section. Finally, the above estimation, together with Lemma \ref{LemmaDegenerateContribution}, \eqref{BoundGenericTerm}, \eqref{Parseval}, \eqref{Eq1} and \eqref{Originalquantity} give the conclusion of Proposition \ref{PropositionL2norm}.

%%%%%%%%%%%%%%%%%%%%%%%%%%%%%%%%%%%%%%%%%%%%
%%%%%%%%%%%%%%%%%%%%%%%%%%%%%%%%%%%%%%%%%%%%
%%%%%%%%%%%%%%%%%%%%%%%%%%%%%%%%%%%%%%%%%%%%
%%%%%						SUBSECTION : A DEGENERATE TERM
%%%%%%%%%%%%%%%%%%%%%%%%%%%%%%%%%%%%%%%%%%%%
%%%%%%%%%%%%%%%%%%%%%%%%%%%%%%%%%%%%%%%%%%%%
%%%%%%%%%%%%%%%%%%%%%%%%%%%%%%%%%%%%%%%%%%%%
\subsection{A degenerate term}\label{SectionDegenerate} \noindent We let $\pi_i$ and $\varphi_i$ be as before and set $\pi_3=1\boxplus 1$ with $\varphi_3=\Erm(f_3(0))$ be the unitary Eisenstein series associated to \eqref{DefinitionFlat}.
Given two small non-zero complex numbers $s,t$ with $|s|\neq |t|$, we wish to understand the value at the point 
\begin{equation}\label{Point}
\mathbf{P}(\lmf):= \varphi_1\otimes\varphi_2^{\lmf}\otimes\varphi_3(s)\otimes\varphi_3(t)^{\lmf}
\end{equation}
 of the $\GL_2(\Abb_\F)$-invariant linear functional
$$\mathscr{H}_{s,t} : \pi_1\otimes\pi_2\otimes\pi_3(s)\otimes\pi_3(t)\longrightarrow \Cbb$$
defined by the rule
$$\mathscr{H}_{s,t}(\phi_1,\phi_2,\phi_3,\phi_4) :=\int_{\X}^{\reg} \phi_1\phi_2\Ecal(\phi_3\phi_4).$$

The map $\mathscr{H}_{s,t}$ is analized in details in \cite[§ 5.2.9]{subconvexity}. In fact, expanding the constant term in the product $\varphi_3(s)\varphi_3(t)^\lmf$, we see that $\Ecal(\varphi_3(s)\varphi_3(t)^\lmf)$ is made of four Eisenstein series, each of them induced by one of the following section 
$$f_3(s)f_3(t)^{\lmf}\  ,  \ f_3(t)^{\lmf}\mathrm{M}f_3(s) \ , \ f_3(s)\mathrm{M}f_3(t)^{\lmf}\ , \ \mathrm{M}f_3(s)\mathrm{M}f_3(t)^{\lmf}.$$
The operator $\mathrm{M}: \pi_3(s)\rightarrow \pi_3(-s)$, $f\mapsto \int_{x\in \Abb_\F}f(wn(x)g)dx$ is the standard intertwiner. By \cite[§ 4.1.8]{subconvexity}, it factorizes as $\mathrm{M}=c\prod_v \mathrm{M}_v$ with $|c|=1$ and $\mathrm{M}_v : \pi_{3,v}(s)\rightarrow \pi_{3,v}(-s)$ has the property that for every place $v$ not dividing the discriminant of $\F$, it takes the spherical vector with value $1$ on $\Krm_v$ to the spherical vector with value $1$ on $\Krm_v$. In any cases, it preserves norms up to a scalar which depends only on $\psi_v$ (note that $\psi$ is fixed by the discriminant of $\F$ by § \ref{Discri}). It follows that
$$|\mathscr{H}_{s,t}(\mathbf{P}(\lmf))|\leqslant \sum_{\pm, \pm } \left|\mathscr{H}_{s,t}^{\pm,\pm}(\mathbf{P}(\lmf))\right|,$$
with
\begin{equation}\label{FactoHscr}
\mathscr{H}_{s,t}^{\pm,\pm}(\mathbf{P}(\lmf))=\frac{\Lambda_\F(1\pm 2s)\Lambda_\F(1\pm 2t)}{\Lambda_\F(2+2(\pm s\pm t))}\Lambda(\pi_1\otimes\pi_2\otimes|\cdot|^{\pm s+\pm t},1)\prod_{v}\mathscr{H}_{s,t,v}^{\pm,\pm}(\lmf_v),
\end{equation}
where for any finite $v$, $\lmf_v$ is the extended ideal $\lmf=\pmf^n$ in the complete ring $\Ocal_{\F_v}$ and the local factors are defined by
\begin{equation}
\mathscr{H}_{s,t,v}^{+,+}(\lmf_v):= \zeta_{\F_v}(2+2(s+t))\frac{\int_{\Nrm(\F_v)\setminus \PGL_2(\F_v)}\Wrm_{1,v}\Wrm_{2,v}^{\lmf_v} f_{3,v}(s)f_{3,v}(t)^{\lmf_v}}{L(\pi_{1,v}\otimes \pi_{2,v}\otimes |\cdot|^{s+t},1)\zeta_{\F_v}(1+2t)\zeta_{\F_v}(1+2s)}.
\end{equation}
The other terms $\mathscr{H}_v^{+,-},\mathscr{H}_v^{-,+}$ and $\mathscr{H}_v^{-,-}$ are defined similarly, by introducing the operator $\mathrm{M}_v$ in front of the $f_{3,v}$'s and by changing $s$ to $-s$ or $t$ to $-t$ as in \eqref{FactoHscr}. Because of our normalization, we have for any finite $v$ prime with the discriminant and $\pmf$
$$\mathscr{H}_{s,t,v}^{\pm,\pm}(\lmf_v)=1.$$
Moreover, for $v<\infty$ and $v\nmid \Delta_\F$
$$\mathscr{H}_{s,t,v}^{-,+}=\mathscr{H}_{-s,t,v}^{+,+} \ , \ \mathscr{H}_{s,t,v}^{+,-}=\mathscr{H}_{s,-t}^{+,+} \ , \ \mathscr{H}_{s,t}^{-,-}=\mathscr{H}_{-s,-t}^{+,+}.$$
We conclude that for small $\varepsilon>0$ and for $s,t$ such that $|s|=2|t|=\varepsilon$, we have 
\begin{equation}\label{LocalPeriod}
\mathscr{H}_{s,t}\left( \mathbf{P}(\lmf)\right)\ll_{\F,\varepsilon} \max_{\substack{|s|=\varepsilon \\ |t|=\varepsilon/2}}\left|\int_{\Nrm(\F_\pmf)\setminus \PGL_2(\F_\pmf)}\Wrm_{1,\pmf}\Wrm_{2,\pmf}^{\lmf}f_{3,\pmf}(s)f_{3,\pmf}(t)^{\lmf}\right|,
\end{equation}
where we recall that the bound depends also on $\pi_1,\pi_2$ and the choice of the infinite datas $\varphi_{1,\infty},\varphi_{2,\infty}$ and $f_{3,\infty}.$ Using Proposition \ref{PropositionLocal}, we obtain :
%%%%%%%%%%%%%%%%%%%%%%%%%%%%%%%%%%%%%%%%%%%%
%%%%%				PROPOSITION DEGENERATE
%%%%%%%%%%%%%%%%%%%%%%%%%%%%%%%%%%%%%%%%%%%%
\begin{proposition}\label{PropositionDegenerate} Let $\varepsilon>0$ be a small real number. Then for all $s,t\in\Cbb$ with $|s|=2|t|=\varepsilon$, we have
$$\mathscr{H}_{s,t}\left(\mathbf{P}(\lmf)\right)\ll_{\F,\pi_1,\pi_2,\varepsilon} \ell^{-1+\varepsilon}.$$
\end{proposition}
\begin{remq}\label{Propositionsmalldeformation} Proposition \ref{PropositionDegenerate} stills valid for a small deformation of $\pi_1$ and $\pi_2$, which is actually how degenerate terms appear in our paper (see for example \eqref{Dege1},\eqref{Dege2}).
\end{remq}

%%%%%%%%%%%%%%%%%%%%%%%%%%%%%%%%%%%%%%%%%%%%
%%%%%%%%%%%%%%%%%%%%%%%%%%%%%%%%%%%%%%%%%%%%
%%%%%%%%%%%%%%%%%%%%%%%%%%%%%%%%%%%%%%%%%%%%
%%%%%			A GENERIC TERM
%%%%%%%%%%%%%%%%%%%%%%%%%%%%%%%%%%%%%%%%%%%%
%%%%%%%%%%%%%%%%%%%%%%%%%%%%%%%%%%%%%%%%%%%%
%%%%%%%%%%%%%%%%%%%%%%%%%%%%%%%%%%%%%%%%%%%%
\subsection{A generic term}\label{SectionGeneric} The combination of Propositions \ref{PropositionL2norm} and \ref{PropositionDegenerate} gives the following estimate for a particular generic expansion :
\begin{proposition}\label{PropositionGeneric} Let $\varepsilon>0$. Then the generic expansion
\begin{equation}\label{GenericExpansion}
\begin{split}
 & \sum_{\substack{\pi \ \mathrm{cuspidal} \\ \crm(\pi)| \lmf}}\sum_{\psi\in\Bscr(\pi,\lmf)} \left|\langle \varphi_1\varphi_2^{\lmf},\psi\rangle \right|^2 \\ 
 + & \ \sum_{\substack{\chi\in\widehat{\F^\times\setminus\Abb_\F^{1}} \\ \crm(\chi)^2|\lmf}}\int_{-\infty}^\infty\sum_{\psi_{it}\in\Bscr(\chi,\chi^{-1},it,\lmf)}\left|\langle\varphi_1\varphi_2^{\lmf},\Erm(\psi_{it})\rangle_\reg\right|^2  \frac{dt}{4\pi} ,
\end{split}
\end{equation}
is bounded, up to a constant depending on $\pi_1,\pi_2,\F$ and $\varepsilon$, by $\ell^{\varepsilon}$.
\end{proposition}
\begin{proof}
Assume first that both $\pi_i$ are cuspidal. Then the expansion \eqref{GenericExpansion} is equal to 
$$\left|\left|\varphi_1\varphi_2^{\lmf}\right|\right|^2_{L^2}-V_{\F}^{-1}\sum_{\chi^2=1}\left| \left\langle \varphi_1\varphi_2^\lmf, \varphi_{\chi}\right\rangle\right|^2.$$
The $L^2$-norm is bounded by $1$ by Proposition \ref{PropositionL2}. The one-dimensional contribution is zero if $\pi_1$ is not isomorphic to a quadratic twist of $\pi_2$. Otherwise, since the $\pi_i$'s are unramified at the finite places, there exists at most one quadratic $\chi$ such that $\pi_1\simeq \pi_2\otimes \chi$ (see \cite[p. 3]{rama}). For such a $\chi$, we have using the identity \eqref{RelationHecke} and the temperedness
\begin{equation}\label{BoundIntegral}
\left| \int_{\X}\varphi_1\overline{\varphi}_2^{\lmf}\varphi_\chi\right| \leqslant \zeta_{\F_{\pmf}}(1)\frac{n+1}{\ell^{1/2}}||\varphi_1||_{L^2}||\varphi_2||_{L^2},
\end{equation}
which gives the result in the case where $\pi_1$ and $\pi_2$ are cuspidal. 

Assume now that at least one of the $\pi_i$'s is Eisenstein. Then returning to Section \ref{SectionDeformation} and applying Theorem \ref{TheoremParseval} to \eqref{PDeforme} with $\Phi_1(s)=\varphi_1(s)\varphi_2^{\lmf}(s/2)$ and $\Phi_2(s)=\varphi_1(\sbar/3)\varphi_2^\lmf(\sbar/4)$ in \eqref{DefinitionPhi}, we first oberve that there is no one-dimensional contribution. Indeed, this is because of the disjointness of the exponents. Hence by the analysis made in the preceding two sections, the expansion \eqref{GenericExpansion} is equal to
$$\Pscr(\lmf;0)+\mathrm{additional \ terms}.$$
By Lemma \ref{LemmaDegenerateContribution}, this additional contribution is bounded by
\begin{equation}\label{Dege1}
\max_{|s|=\varepsilon}\left\{\left|\left\langle \Phi_1(s),\Ecal\left(\Phi_2(s)\right)\right\rangle_{\mathrm{reg}}\right|+\left|\left\langle \Ecal\left(\Phi_1(s)\right),\Phi_2(s)\right\rangle_{\mathrm{reg}}\right|\right\}.
\end{equation}
Hence the conclusion follows from Proposition \ref{PropositionL2} for $\Pscr(\lmf;0)$, Proposition \ref{PropositionDegenerate} and Remark \ref{Propositionsmalldeformation} for the above degenerate term.
\end{proof}

%
%
%
%%%%%%%%%%%%%%%%%%%%%%%%%%%%%%%%%%%%%%%%%%%%%
%%%%%%%%%%%%%%%%%%%%%%%%%%%%%%%%%%%%%%%%%%%%%
%%%%%%%%%%%%%%%%%%%%%%%%%%%%%%%%%%%%%%%%%%%%%
%%%%%%%%%%%%%%%%%%%%%%%%%%%%%%%%%%%%%%%%%%%%%
%%%%%%			SECTION : SYMMETRIC PERIOD
%%%%%%%%%%%%%%%%%%%%%%%%%%%%%%%%%%%%%%%%%%%%%
%%%%%%%%%%%%%%%%%%%%%%%%%%%%%%%%%%%%%%%%%%%%%
%%%%%%%%%%%%%%%%%%%%%%%%%%%%%%%%%%%%%%%%%%%%%
%%%%%%%%%%%%%%%%%%%%%%%%%%%%%%%%%%%%%%%%%%%%%
\section{A Symmetric Period}\label{SectionSymmetric}
Let $\pi_1,\pi_2$ and $\varphi_i\in\pi_i$ as in Section \ref{SectionEstimation}. We take $\qmf$ a squarefree ideal of $\Ocal_\F$ and $\lmf$ an integral ideal of the form $\pmf^n$ with $n\in\Nbb$ and $\pmf\in\mathrm{Spec}(\Ocal_\F)$ coprime with $\qmf$. We write $q,p,\ell$ for the norms of $\qmf,\pmf$ and $\lmf$ respectively. We also adopt the convention that all $\ll$ involved in this section depend implicitly on the infinite datas $\varphi_{i,\infty}$ (c.f. Remark \ref{RemarkInfinite}). Setting
\begin{equation}\label{DefinitionPhi}
\Phi=\varphi_1\varphi_2^{\qmf},
\end{equation}
we consider the period
\begin{equation}\label{ThePeriod}
\Pscr_\qmf(\lmf,\Phi,\Phi) := \int_{\X}^{\mathrm{reg}}\Trm_{\lmf}(\Phi) \overline{\Phi}=\left\langle \Trm_{\lmf}(\Phi),\Phi \right\rangle_{\mathrm{reg}}.
\end{equation}

%%%%%%%%%%%%%%%%%%%%%%%%%%%%%%%%%%%%%%%%%%%%%
%%%%%%%%%%%%%%%%%%%%%%%%%%%%%%%%%%%%%%%%%%%%%
%%%%%%%%%%%%%%%%%%%%%%%%%%%%%%%%%%%%%%%%%%%%%
%%%%%				EXPANSION IN THE Q-ASPECT
%%%%%%%%%%%%%%%%%%%%%%%%%%%%%%%%%%%%%%%%%%%%%
%%%%%%%%%%%%%%%%%%%%%%%%%%%%%%%%%%%%%%%%%%%%%
%%%%%%%%%%%%%%%%%%%%%%%%%%%%%%%%%%%%%%%%%%%%%
\subsection{Expansion in the $\qmf$-aspect}
Noting that $\Phi_1=\varphi_1\varphi_2^\qmf$ is an automorphic form invariant under the group $\Krm_0(\qmf)$, we apply Plancherel formula to the regularized inner product \eqref{ThePeriod} in the space of forms of level $\qmf$. Using the same deformation argument of Section \ref{SectionDeformation} and the analysis of the additional contribution provided by Lemma \ref{LemmaDegenerateContribution} gives the following decomposition of the considered period
\begin{equation}\label{Expansion1}
\Pscr_{\qmf}(\lmf,\Phi,\Phi)= \Gscr_{\qmf}(\lmf,\Phi,\Phi)+ \Cscr_1 +\mathrm{additional \ terms},
\end{equation}
where the generic part is given by
\begin{equation}\label{GenericExpansionQ}
\begin{split}
 \Gscr_{\qmf}(\lmf,\Phi,\Phi)= & \sum_{\substack{\pi \ \mathrm{cuspidal} \\ \crm(\pi)| \qmf}}\lambda_\pi(\lmf)\sum_{\psi\in\Bscr(\pi,\qmf)} \left|\langle \varphi_1\varphi_2^{\qmf},\psi\rangle \right|^2 \\ 
 + & \ \sum_{\substack{\chi\in\widehat{\F^\times\setminus\Abb_\F^{1}} \\ \crm(\chi)=\Ocal_\F}}\int_{-\infty}^\infty \lambda_{\chi,it}(\lmf)\sum_{\psi_{it}\in\Bscr(\chi,\chi^{-1},it,\qmf)}\left|\langle\varphi_1\varphi_2^{\qmf},\Erm(\psi_{it})\rangle_\reg\right|^2  \frac{dt}{4\pi}.
\end{split}
\end{equation}
Here $\Cscr_1$ is the one-dimensional contribution which appears only if both $\pi_1$ and $\pi_2$ are cuspidal and there exists a quadratic character $\chi$ of $\F^\times\setminus \Abb_\F^\times$ such that $\pi_1\simeq \pi_2\otimes\chi$ ($\chi=1$ if $\pi_1=\pi_2$ for example). In this case we obtain since $\chi$ is also unramified at the finite places and by \eqref{RelationHecke} (recall also that $\qmf$ is squarefree)
\begin{alignat}{1}
\Cscr_1= & \ V_{\F}^{-1}\langle \Trm_{\lmf}(\Phi),\varphi_\chi\rangle\langle \varphi_\chi,\Phi\rangle = \frac{\chi(\lmf)\deg(\Trm_\lmf)}{V_\F}\left|\left\langle \varphi_1\varphi_2^{\qmf},\varphi_\chi\right\rangle\right|^2 \nonumber\\
= & \ \lambda_{\pi_2}(\qmf)^2\frac{\zeta_{\qmf}(2)^2\chi(\lmf)\deg(\Trm_\lmf)}{q\zeta_{\qmf}(1)^2 V_\F}\left|\left\langle \varphi_1\varphi_2,\varphi_\chi\right\rangle\right|^2\label{q-error},
 \end{alignat}
where $\zeta_\qmf(s)=\prod_{v|\qmf}\zeta_{\F_v}(s)$ is the partial Dedekind zeta function and the degree of the Hecke operator $\Trm_\lmf$ is defined by 
\begin{equation}
\deg(\Trm_\lmf) := \frac{1}{\ell^{1/2}} \sum_{0\leqslant k\leqslant \frac{n}{2}}\gamma_{n-2k} \leqslant \ell^{1/2}\zeta_{\F_\pmf}(1).
\end{equation}
Hence since $\pi_2$ is tempered at all the places dividing $\qmf$, we get 
$$\Cscr_1 \ll_{\pi_1,\pi_2,\F,\varepsilon} (\ell q)^\varepsilon \frac{\ell^{1/2}}{q}.$$
The additional terms appear only if the two representations are Eisenstein. In this case, for $\Phi_1(s)=\varphi_1(s)\varphi_2(s/2)^\qmf$ and $\Phi_2(s)=\varphi_1(\sbar/3)\varphi_2(\sbar/4)^\qmf$, these terms are bounded by 
\begin{equation}\label{Dege2}
\max_{|s|=\varepsilon}\left\{\left|\left\langle \Phi_1(s),\Trm_{\lmf}\left[\Ecal\left(\Phi_2(s)\right)\right]\right\rangle_{\mathrm{reg}}\right| \right\} \ll_{\pi_1,\pi_2,\F,\varepsilon}(\ell q)^\varepsilon \frac{\ell^{1/2}}{q},
\end{equation}
by Lemma \ref{LemmaDegenerateContribution}, the analysis made in Section \ref{Paragraph} (for the $\lmf$-aspect) and Proposition \ref{PropositionDegenerate} for the $\qmf$-aspect. Hence we conclude with
\begin{equation}\label{FirstRelation}
\Pscr_{\qmf}(\lmf,\Phi,\Phi)=\Gscr_\qmf(\lmf,\Phi,\Phi)+O_{\pi_1,\pi_2,\F,\varepsilon}\left( (\ell q)^\varepsilon \frac{\ell^{1/2}}{q}\right),
\end{equation}
where we recall that $\Phi$ is defined by \eqref{DefinitionPhi}.

%%%%%%%%%%%%%%%%%%%%%%%%%%%%%%%%%%%%%%%%%%%%
%%%%%%%%%%%%%%%%%%%%%%%%%%%%%%%%%%%%%%%%%%%%
%%%%%%%%%%%%%%%%%%%%%%%%%%%%%%%%%%%%%%%%%%%%
%%%%%			SUBSECTION : THE SYMMETRIC RELATION
%%%%%%%%%%%%%%%%%%%%%%%%%%%%%%%%%%%%%%%%%%%%
%%%%%%%%%%%%%%%%%%%%%%%%%%%%%%%%%%%%%%%%%%%%
%%%%%%%%%%%%%%%%%%%%%%%%%%%%%%%%%%%%%%%%%%%%
\subsection{The symmetric relation} The symmetric relation is obtained by grouping differently the vectors $\varphi_i$ : in the period $\Pscr_{\qmf}(\lmf,\Phi,\Phi)$, we first use the relation \eqref{RelationHecke} to expand the Hecke operator $\Trm_\lmf$. Secondly we do the same, but on the reverse way, for the translation by the matrix $\left(\begin{smallmatrix} 1 & \\ & \varpi_\qmf\end{smallmatrix}\right)$, making this time the operator $\Trm_{\qmf}$ appears. We thus infer the following symmetric relation :
\begin{equation}\label{Symmetry1}
q^{1/2}\frac{\zeta_\qmf(1)}{\zeta_\qmf(2)}\Pscr_\qmf(\lmf,\Phi,\Phi)= \frac{1}{\ell^{1/2}}\sum_{0\leqslant k\leqslant \frac{n}{2}}\gamma_{n-2k}\Pscr_{\pmf^{n-2k}}(\qmf,\Psi_1,\Psi_2), 
\end{equation}
with
\begin{equation}\label{Psi}
\Psi_1= \overline{\varphi}_1\varphi_1^{\pmf^{n-2k}} \ \ \mathrm{and} \ \ \Psi_2=\varphi_2\overline{\varphi}_2^{\pmf^{n-2k}}.
\end{equation}

Periods $\Pscr_{\pmf^{n-2k}}(\qmf,\Psi_1,\Psi_2)$ admit a similar expansion as \eqref{Expansion1}, but this time over representations of conductor dividing $\pmf^{n-2k}$. We are now familiar with the spectral decomposition 
$$\Pscr_{\pmf^{n-2k}}(\qmf,\Psi_1,\Psi_2)=\Gscr_{\pmf^{n-2k}}(\qmf,\Psi_1,\Psi_2)+ \mathscr{C}_2(k) +\mathrm{additional \ terms},$$
where $\Gscr_{\pmf^{n-2k}}(\qmf,\Psi_1,\Psi_2)$ is the generic part, $\mathscr{C}_2$ is non-zero only when both $\pi_1$ and $\pi_2$ are cuspidal; of course in this case there are no additional terms. We have in the same spirit of the previous section (see \eqref{Dege2})
\begin{equation}\label{Addition1}
\mathrm{additional \ terms }\ll_{\pi_1,\pi_2,\F,\varepsilon} (\ell q)^\varepsilon \frac{q^{1/2}}{p^{n-2k}},
\end{equation}
and by definition
\begin{equation}\label{GenericExpansionP}
\begin{split}
  & \Gscr_{\pmf^{n-2k}}(\qmf,\Psi_1,\Psi_2)= \sum_{\substack{\pi \ \mathrm{cuspidal} \\ \crm(\pi)| \pmf^{n-2k}}}\lambda_\pi(\qmf)\sum_{\psi\in\Bscr(\pi,\pmf^{n-2k})} \langle \Psi_1,\psi\rangle\langle\psi, \Psi_2\rangle \\ 
 + & \ \sum_{\substack{\chi\in\widehat{\F^\times\setminus\Abb_\F^{1}} \\ \crm(\chi)^2|\pmf^{n-2k}}}\int_{-\infty}^\infty \lambda_{\chi,it}(\qmf)\sum_{\psi_{it}\in\Bscr(\chi,\chi^{-1},it,\pmf^{n-2k})}\langle\Psi_1,\Erm(\psi_{it})\rangle_\reg\langle\Erm(\psi_{it}),\Psi_2\rangle_\reg \frac{dt}{4\pi}.
\end{split}
\end{equation}
Assuming that $\pi_1$ and $\pi_2$ are cuspidal, there are no additional terms and the constant $\mathscr{C}_2(k)$ is equal to 
\begin{equation}\label{ConstantC2}
\mathscr{C}_2(k)=V_{\F}^{-1}q^{1/2}\frac{\zeta_\qmf(1)}{\zeta_\qmf(2)}\prod_{i=1}^2\left( \int_{\X}\varphi_i\overline{\varphi}_i^{\pmf^{n-2k}}\right),
\end{equation}
with by \eqref{BoundIntegral}
$$\left| \int_{\X}\varphi_i\overline{\varphi}_i^{\pmf^{n-2k}}\right|\leqslant \zeta_{\F_\pmf}(1)\frac{n-2k+1}{p^{\frac{n-2k}{2}}}||\varphi_i||_{L^2}^2 \Longrightarrow \Cscr_2(k)\ll_{\varepsilon,\pi_1,\pi_2,\F}(q\ell)^\varepsilon \frac{q^{1/2}}{p^{n-2k}}.$$
The total complete constant term is obtained after summing over $0\leqslant k\leqslant n/2$ as in \eqref{Symmetry1}, i.e. 
\begin{equation}\label{GlobalConstantTerm}
\Cscr_2(\pi_1,\pi_2):= \frac{1}{\ell^{1/2}}\sum_{0\leqslant k\leqslant \frac{n}{2}}\gamma_{n-2k}\Cscr_2(k)
\end{equation}
with the following upper bound
\begin{equation}\label{BoundConstantTerm}
\mathscr{C}_2(\pi_1,\pi_2) \ll_{\varepsilon,\pi_1,\pi_2,\F} (q\ell)^\varepsilon \frac{q^{1/2}}{\ell^{1/2}}.
\end{equation}
Assembling \eqref{FirstRelation},\eqref{Symmetry1},\eqref{Addition1} and \eqref{GlobalConstantTerm}, we find the following reciprocity relation between the two generic parts 
\begin{equation}\label{ReciprocityRelation}
\begin{split}
&q^{1/2}\frac{\zeta_\qmf(1)}{\zeta_\qmf(2)}\Gscr_{\qmf}(\lmf,\Phi,\Phi)=  \frac{1}{\ell ^{1/2}}\sum_{0\leqslant k\leqslant \frac{n}{2}}\gamma_{n-2k}\Gscr_{\pmf^{n-2k}}(\qmf,\Psi_1,\Psi_2) \\ + & \left\{ \begin{array}{llc} \Cscr_2(\pi_1,\pi_2)+O_{\pi_1,\pi_2,\F,\varepsilon}\left( (q\ell)^{\varepsilon}\left(\frac{\ell}{q} \right)^{1/2}\right) & \mathrm{if} & \pi_1,\pi_2 \ \mathrm{are \ cuspidal} \\ & & \\O_{\pi_1,\pi_2,\F,\varepsilon}\left((\ell q)^\varepsilon \left( \left(\frac{q}{\ell}\right)^{1/2} +\left(\frac{\ell}{q}\right)^{1/2} \right) \right) & \mathrm{otherwise}. \end{array}\right.
\end{split}
\end{equation}
\begin{remq}\label{RemarkNonTempered}
If $\pi_1$ and $\pi_2$ are cuspidal and $\pi_1$ is isomorphic to a quadratic twist of $\pi_2$, but not necessarily tempered at the finite places, then by \eqref{q-error}, the error term in the second line of \eqref{ReciprocityRelation} has to be replaced by $q^{2\vartheta_{\pi_1}}\left(\tfrac{\ell}{q}\right)^{1/2}$.
\end{remq}
Finally, we can estimate the generic terms on the righthand side simply using the bound $\lambda_\pi(\qmf)\leqslant \tau(\qmf)q^{\vartheta}$, Cauchy-Schwarz inequality and Proposition \ref{PropositionGeneric}, obtaining 
\begin{equation}\label{FinalBound1}
\frac{1}{\ell^{1/2}}\sum_{0\leqslant k\leqslant \frac{n}{2}}\gamma_{n-2k}\Gscr_{\pmf^{n-2k}}(\qmf,\Psi_1,\Psi_2) \ll_{\pi_1,\pi_2,\F,\varepsilon} (\ell q)^\varepsilon \ell^{1/2}q^{\vartheta}.
\end{equation}

%%%%%%%%%%%%%%%%%%%%%%%%%%%%%%%%%%%%%%%%%%%%
%%%%%%%%%%%%%%%%%%%%%%%%%%%%%%%%%%%%%%%%%%%%
%%%%%%%%%%%%%%%%%%%%%%%%%%%%%%%%%%%%%%%%%%%%
%%%%%%%%%%%%%%%%%%%%%%%%%%%%%%%%%%%%%%%%%%%%
%%%%%			CONNECTION WITH TRIPLE L-FUNCTIONS
%%%%%%%%%%%%%%%%%%%%%%%%%%%%%%%%%%%%%%%%%%%%
%%%%%%%%%%%%%%%%%%%%%%%%%%%%%%%%%%%%%%%%%%%%
%%%%%%%%%%%%%%%%%%%%%%%%%%%%%%%%%%%%%%%%%%%%
%%%%%%%%%%%%%%%%%%%%%%%%%%%%%%%%%%%%%%%%%%%%
\subsection{Connection with triple product}\label{Connection}
\noindent We connect in this section the expansion \eqref{GenericExpansionQ} with a moment of the triple product $L(\pi\otimes\pi_1\otimes\pi_2,\tfrac{1}{2})$ over representations $\pi$ of conductor dividing $\qmf$. For such a representation $\pi$, we define
\begin{equation}\label{Lscr}
\Lscr(\pi,\qmf) :=\sum_{\psi \in \Bscr(\pi,\qmf)}\left| \langle \varphi_1\varphi_2^\qmf,\psi\rangle_{\mathrm{reg}}\right|^2,
\end{equation}
where we recall that $\Bscr(\pi,\qmf)$ is an orthonormal basis of the space of $\Krm_0(\qmf)$-vectors in $\pi$. By Proposition \ref{PropositionIntegralRepresentation} and Definition \eqref{CanonicalNorm} of the canonical norm, we have
\begin{equation}\label{FactorizationLcal}
\Lscr(\pi,\qmf)=\frac{C}{2\Delta_\F^{1/2}}f(\pi_\infty)\frac{L(\pi\otimes\pi_1\otimes\pi_2,\tfrac{1}{2})}{\Lambda^*(\pi,\mathrm{Ad},1)} \ell(\pi,\qmf),
\end{equation}
where 
\begin{equation}\label{ValueC}
C= \left\{ \begin{array}{lll}  2\Lambda_\F(2) & \ifm & \pi,\pi_1,\pi_2 \ \mathrm{are \ cuspidal} \\
 & & \\
 
2\Lambda_\F^*(1) & \ifm & \pi \ \mathrm{is \ Eisenstein \ and \ nonsingular} \\
 & & \\
1 & \mathrm{otherwise}.
\end{array} \right.
\end{equation}
If we identify $\pi\simeq \otimes_v\pi_v$, then $\ell(\pi,\qmf)=\prod_{v|\qmf}\ell_v$ and the local factors $\ell_v$ are given, for $v|\qmf$, by 
\cite[(3.9) \& Section 4]{raphael}
\begin{equation}\label{ellv}
\ell_v = \left\{ \begin{array}{lcl}
q_v^{-1}\frac{\zeta_{\F_v}(2)}{\zeta_{\F_v}(1)} & \ifm & \crm(\pi_v) = 1 \\
& & \\
\kappa_v & \ifm & \crm(\pi_v) =0,
\end{array}\right.
\end{equation}
where $\kappa_v$ depends on the local datas $\pi_v,\pi_{1,v}$ and $\pi_{2,v}$ and can be computed explicitly using \cite[Section 4.2]{raphael}. We obtain 
\begin{equation}\label{kappav}
0<\kappa_v \ll q_v^{-1+2\max(\vartheta_{\pi_1},\vartheta_{\pi_2})}.
\end{equation}

%%%%%%%%%%%%%%%%%%%%%%%%%%%%%%%%%%%%%%%%%%%%
%%%%%%%%%%%%%%%%%%%%%%%%%%%%%%%%%%%%%%%%%%%%
%%%%%%%%%%%%%%%%%%%%%%%%%%%%%%%%%%%%%%%%%%%%
%%%%%%%%%%%%%%%%%%%%%%%%%%%%%%%%%%%%%%%%%%%%
%%%%%			INTERLUDE ON ARCHIMEDEAN TEST FUNCTION
%%%%%%%%%%%%%%%%%%%%%%%%%%%%%%%%%%%%%%%%%%%%
%%%%%%%%%%%%%%%%%%%%%%%%%%%%%%%%%%%%%%%%%%%%
%%%%%%%%%%%%%%%%%%%%%%%%%%%%%%%%%%%%%%%%%%%%
%%%%%%%%%%%%%%%%%%%%%%%%%%%%%%%%%%%%%%%%%%%%
\subsection{Interlude on the Archimedean function $f(\pi_\infty)$}\label{SectionInterlude}
The Archimedean function $f(\pi_\infty)$ appearing in the factorization \eqref{FactorizationLcal} is given by (c.f. \cite[eq. (3.10)]{raphael}) 
\begin{equation}\label{Definitionfinfty}
f(\pi_\infty):= \sum_{\varphi_\infty \in \Bscr(\pi_\infty)}I_\infty(\varphi_\infty\otimes\varphi_{1,\infty}\otimes\varphi_{2,\infty})L(\pi_\infty\otimes\pi_{1,\infty}\otimes\pi_{2,\infty},\tfrac{1}{2}),
\end{equation}
where the local period $I_\infty$ is defined in \eqref{DefinitionNormalizedMatrixCoefficient}. The function $f(\pi_\infty)$ is non-negative and depends on the infinite factors $\pi_{1,\infty}$ and $\pi_{2,\infty}$ and more precisely, on the choice of test vectors $\varphi_{i,\infty}\in \pi_{i,\infty}$ and the orthonormal basis $\Bscr(\pi_\infty)$. For our application, it is fundamental that $f$ satisfies at least the following property : given $\pi=\pi_\infty\otimes\pi_{\mathrm{fin}}$ a unitary automorphic representation of $\PGL_2(\Abb_\F)$, there exists $\varphi_{i,\infty}\in\pi_{i,\infty}$, $i=1,2$ having norm $1$ and a basis $\Bscr(\pi_\infty)$ such that $f(\pi_\infty)$ is bounded below by a power of the Archimedean conductor $\crm(\pi_\infty).$ It is a result of Michel and Venkatesh \cite[Proposition 3.6.1]{subconvexity} that such a choice exists when the $\pi_i$'s satisfy the following hypothesis :
\begin{Hyp}\label{Hyp}
Assume that for all Archimedean place $v|\infty$, either $\pi_{1,v}$ or $\pi_{2,v}$ is a principal series representation. Then for any $\varepsilon > 0$, there exists a positive constant $C(\pi_{1,\infty},\pi_{2,\infty},\varepsilon)$ such that we have the lower bound
\begin{equation}\label{LowerBound}
f(\pi_\infty)\geqslant \frac{C(\pi_{1,\infty},\pi_{2,\infty},\varepsilon)}{\crm(\pi_\infty)^{1+\varepsilon}}.
\end{equation}
\end{Hyp}

\begin{remq} Observe that if one of the $\pi_i$'s is Eisenstein, then $\pi_{1,v}$ is a principal series for every $v|\infty$, so the condition mentionned above is satisfied. Futheremore, $\varphi_{1,v}$ is constructed by choosing suitable Schwartz functions $\Psi_{i,v}$ in \eqref{DefinitionFlat} (see \cite[Definition 3.6.4]{subconvexity}).
\end{remq}

\begin{remq}\label{RemarkInfty} Fixing $\pi=\pi_\infty\otimes\pi_{\mathrm{fin}}$ and assuming that $\pi_1,\pi_2$ satisfy Hypothesis \ref{Hyp}, then all implied constants involved in Sections \ref{SectionEstimation} and \ref{SectionSymmetric} depend no more on $\varphi_{i,\infty}$, but only on $\pi_{\infty}$ (of course also on $\pi_1,\pi_2$). It should be possible with more works to be more explicit about this dependency (e.g. polynomially on the datas of $\pi_\infty,\pi_{i,\infty}$), but we leave that aspect aside in this work, as well as an explicit description of the function $f(\pi_\infty)$ for varying vectors $\varphi_{i,\infty}$.
\end{remq}

%%%%%%%%%%%%%%%%%%%%%%%%%%%%%%%%%%%%%%%%%%%%
%%%%%%%%%%%%%%%%%%%%%%%%%%%%%%%%%%%%%%%%%%%%
%%%%%%%%%%%%%%%%%%%%%%%%%%%%%%%%%%%%%%%%%%%%
%%%%%%%%%%%%%%%%%%%%%%%%%%%%%%%%%%%%%%%%%%%%
%%%%%				PROOF THEOREM 1
%%%%%%%%%%%%%%%%%%%%%%%%%%%%%%%%%%%%%%%%%%%%
%%%%%%%%%%%%%%%%%%%%%%%%%%%%%%%%%%%%%%%%%%%%
%%%%%%%%%%%%%%%%%%%%%%%%%%%%%%%%%%%%%%%%%%%%
%%%%%%%%%%%%%%%%%%%%%%%%%%%%%%%%%%%%%%%%%%%%
\subsection{Proof of Theorem \ref{Theorem1}} We define the cuspidal part by
\begin{equation}\label{CuspidalPart}
\mathscr{C}(\pi_1,\pi_2,\qmf,\lmf):= C\sum_{\substack{\pi \ \mathrm{cuspidal} \\ \crm(\pi)|\qmf}}\lambda_\pi(\lmf)\frac{L(\pi\otimes\pi_1\otimes\pi_2,\tfrac{1}{2})}{\Lambda(\pi,\mathrm{Ad},1)}f(\pi_\infty)\Lcal(\pi,\qmf),
\end{equation}
where $\Lcal(\pi,\qmf)$ is the normalized version of $\ell(\pi,\qmf)$, i.e.
\begin{equation}\label{Lcal}
\Lcal(\pi,\qmf)=q\frac{\zeta_\qmf(1)}{\zeta_\qmf(2)}\ell(\pi,\qmf).
\end{equation}
For the continuous part, we denote by $\pi_{\omega}(it)$ the principal series $\omega|\cdot|^{it}\boxplus \omegabar|\cdot|^{-it}.$ We then set 
\begin{equation}\label{ContinuousPart}
\begin{split}
\mathscr{E}(\pi_1,\pi_2,\qmf,\lmf):= C\sum_{\substack{\omega\in\widehat{\F^\times\setminus \Abb_\F^{1}} \\ \crm(\omega)=\Ocal_\F}}&\int_{-\infty}^\infty \lambda_{\pi_\omega(it)}(\lmf)f(\pi_{\omega_\infty}(it))\Lcal(\pi_\omega(it),\qmf) \\ \times & \frac{L(\pi_1\otimes\pi_2\otimes\omega,\tfrac{1}{2}+it)L(\pi_1\otimes\pi_2\otimes\omegabar,\tfrac{1}{2}-it)}{\Lambda^*(\pi_\omega(it),\mathrm{Ad},1)} \frac{dt}{4\pi},
\end{split}
\end{equation}
and
\begin{equation}\label{DefinitionMoment}
\Mscr(\pi_1,\pi_2,\qmf,\lmf) :=\mathscr{C}(\pi_1,\pi_2,\qmf,\lmf)+ \mathscr{E}(\pi_1,\pi_2,\qmf,\lmf).
\end{equation}
Observing now that 
\begin{equation}\label{Scale}
\Mscr(\pi_1,\pi_2,\qmf,\lmf)=q\frac{\zeta_\qmf (1)}{\zeta_\qmf (2)}\mathscr{G}_\qmf(\lmf,\Phi,\Phi),
\end{equation}
where $\mathscr{G}_\qmf(\lmf,\Phi,\Phi)$ is the generic expansion \eqref{GenericExpansionQ}, we obtain, by \eqref{BoundConstantTerm},\eqref{ReciprocityRelation} and \eqref{FinalBound1} the conclusion of Theorem \ref{Theorem1}.

%%%%%%%%%%%%%%%%%%%%%%%%%%%%%%%%%%%%%%%%%%%%
%%%%%%%%%%%%%%%%%%%%%%%%%%%%%%%%%%%%%%%%%%%%
%%%%%%%%%%%%%%%%%%%%%%%%%%%%%%%%%%%%%%%%%%%%
%%%%%%%%%%%%%%%%%%%%%%%%%%%%%%%%%%%%%%%%%%%%
%%%%%				PROOF THEOREM 3 AND 4
%%%%%%%%%%%%%%%%%%%%%%%%%%%%%%%%%%%%%%%%%%%%
%%%%%%%%%%%%%%%%%%%%%%%%%%%%%%%%%%%%%%%%%%%%
%%%%%%%%%%%%%%%%%%%%%%%%%%%%%%%%%%%%%%%%%%%%
%%%%%%%%%%%%%%%%%%%%%%%%%%%%%%%%%%%%%%%%%%%%
\subsection{Proof of Theorems \ref{Theorem3} and \ref{Theorem4}}\label{SectionProof34} We assume in this section that we are in the case where $\pi_1=\pi_2=:\sigma$ and $\varphi_1=\varphi_2=:\varphi$. Then it is possible under this assumption to connect the generic expansion \eqref{GenericExpansionP} to a moment of automorphic $L$-functions. Indeed, assuming that $\lmf$ (of norm $\ell$) is squarefree, i.e. $n=1$ in the above treatment, and given any $\pi$ unitary of conductor dividing $\lmf$, the analogue of \eqref{Lscr} is
$$\widetilde{\mathscr{L}}(\pi,\lmf):= \sum_{\psi\in\Bscr(\pi,\lmf)}\langle \overline{\varphi}\varphi^{\lmf},\psi\rangle\langle \psi,\varphi\overline{\varphi}^{\lmf}\rangle,$$
where $\psi$ could be possibly an Eisenstein series and the inner product as to be understood in this case as the regularized version. The terms inside the sum are almost square of periods; we see that the only obstruction is the complex conjugation which is not at the same place. To remedy this situation, we use exactly the same idea of \cite[Proposition 3.2]{raphael}, obtaining a similar expression than \eqref{FactorizationLcal}, but with an extra local root number at $\lmf$
\begin{equation}\label{Facto2}
\widetilde{\mathscr{L}}(\pi,\pmf)=\frac{C}{2\Delta_\F^{1/2}}\widetilde{f}(\pi_\infty)\varepsilon_{\pi}(\lmf)\frac{L(\pi\otimes\sigma\otimes\sigma,\tfrac{1}{2})}{\Lambda^*(\pi,\mathrm{Ad},1)}\widetilde{\ell}(\pi,\lmf),
\end{equation}
where $\widetilde{f}(\pi_\infty)$ is defined as $f(\pi_\infty)$ in \eqref{Definitionfinfty} with $\varphi_{1,\infty}=\varphi_{\infty}$ and $\varphi_{2,\infty}=\overline{\varphi}_{\infty}$ and $\widetilde{\ell}(\pi,\lmf)$ is a product over the places $v|\lmf$ of local factors $\widetilde{\ell}_v(\pi_v,\lmf_v)$ with $\widetilde{\ell}_v$ as in \eqref{ellv} if $\crm(\pi_v)=1$, or equal to $\widetilde{\kappa}_v$ if $\crm(\pi_v)=0$ with $\widetilde{\kappa}_v$ satisfying the same property \eqref{kappav}.

\vspace{0.1cm}

Observe now that we have the factorization of the $L$-function \cite[eq. (1.5)]{nelson}

\begin{equation}\label{Sym}
L(\pi\otimes\sigma\otimes\sigma,\tfrac{1}{2})= \left\{ \begin{array}{lcl} L(\pi,\tfrac{1}{2})^4 & \mathrm{if} & \sigma=1\boxplus 1 \\ 
 & & \\ 
 L(\mathrm{Sym}^2(\sigma)\otimes \pi,\tfrac{1}{2})L(\pi,\tfrac{1}{2}) & \mathrm{if} & \sigma  \ \mathrm{cuspidal}.
 \end{array}\right.
\end{equation}
Moreover, if $\lmf$ is squarefree, we can easily compute the constant $\Cscr_2$ \eqref{ConstantC2}-\eqref{GlobalConstantTerm} in the case where $\sigma$ is cuspidal :
\begin{equation}\label{ValueC2}
\Cscr_2 = \frac{q^{1/2}\zeta_\qmf (1)}{V_\F \zeta_{\qmf}(2)}||\varphi||_{L^2}^4\times \left\{ \begin{array}{lcc} \lambda_\sigma(\lmf)^2\frac{\zeta_{\lmf}(2)}{p^{1/2}\zeta_{\lmf}(1)} & \mathrm{if} & \lmf \neq 1 \\ 
  & & \\
  1 & \mathrm{if} & \lmf=1, \end{array}\right.
\end{equation}
and by \eqref{Comparition}
$$||\varphi||_{L^2}^4 = 4\Delta_\F \Lambda(\sigma,\mathrm{Ad},1)^2.$$
%%%%%%%%%%%%%%%%%%%%%%%%%%%%%%%%%%%%%%%%%%%%
%%%%%%%%%%%%%%%%%%%%%%%%%%%%%%%%%%%%%%%%%%%%
%%%%%						PROOF THEOREM 3
%%%%%%%%%%%%%%%%%%%%%%%%%%%%%%%%%%%%%%%%%%%%
%%%%%%%%%%%%%%%%%%%%%%%%%%%%%%%%%%%%%%%%%%%%
\begin{proof}[proof of Theorem \ref{Theorem3}] The proof is similar to Theorem \ref{Theorem1}. We first let $\widetilde{\mathscr{C}}(\sigma,\lmf,\qmf)$, $\widetilde{\mathscr{E}}(\sigma,\lmf,\qmf)$ be respectively as in \eqref{CuspidalPart} and \eqref{ContinuousPart}, but with the following modifications :
\begin{enumerate}
\item[$\bullet$] The role of $\qmf$ and $\lmf$ are exchanged;
\item[$\bullet$] The Archimedean test function $f(\pi_\infty)$ is replaced by $\widetilde{f}(\pi_\infty)$;
\item[$\bullet$] The non-Archimedean weight function $\Lcal$ is replaced by $\widetilde{\Lcal}$ (at the places $v|\lmf$), with $\widetilde{\Lcal}$ is the normalized version (as in \eqref{Lcal}) of $\widetilde{\ell}$ defined in the previous paragraph.
\item[$\bullet$] We attach a local root number $\varepsilon_\pi(\lmf)$ to each $L$-function appearing in the spectral expansion. Note that $\varepsilon_\pi(\lmf)=1$ if $\pi$ is unramified.
\end{enumerate}

Setting 
\begin{equation}
\widetilde{\Mscr}(\sigma,\lmf,\qmf):= \widetilde{\mathscr{C}}(\sigma,\lmf,\qmf) + \widetilde{\mathscr{E}}(\sigma,\lmf,\qmf),
\end{equation}
we have as in \eqref{Scale}
$$\widetilde{\Mscr}(\sigma,\lmf,\qmf) = \ell\frac{\zeta_\lmf(1)}{\zeta_\lmf(2)}\mathscr{G}_{\lmf}(\qmf,\Psi_1,\overline{\Psi}_1),$$
where $\mathscr{G}_{\lmf}(\qmf,\Psi_1,\overline{\Psi}_1)$ is the generic expansion defined in \eqref{GenericExpansionP}. The conclusion follows now from the symmetric relation \eqref{ReciprocityRelation} (see also Remark \ref{RemarkNonTempered}) and the computation of the main term \eqref{ValueC2} in the cuspidal case.
\end{proof}
%%%%%%%%%%%%%%%%%%%%%%%%%%%%%%%%%%%%%%%%%%%%
%%%%%%%%%%%%%%%%%%%%%%%%%%%%%%%%%%%%%%%%%%%%
%%%%%.     					PROOF THEOREM 4
%%%%%%%%%%%%%%%%%%%%%%%%%%%%%%%%%%%%%%%%%%%%
%%%%%%%%%%%%%%%%%%%%%%%%%%%%%%%%%%%%%%%%%%%%
\begin{proof}[proof of Theorem \ref{Theorem4}] Assume that $\qmf\in\mathrm{Spec}(\Ocal_\F)$. Applying Theorem \ref{Theorem3} with $\lmf=1$, so $\ell=1$, we obtain (see the value of the main term \eqref{ValueC2} in this case)
$$\frac{1}{q}\Mscr(\sigma,\qmf,1)= \frac{1}{q^{1/2}}\widetilde{\Mscr}(\sigma,1,\qmf)+\frac{4\Delta_\F\Lambda(\sigma,\mathrm{Ad},1)^2}{V_\F}\cdot \frac{\zeta_\qmf(1)}{\zeta_\qmf(2)}+O_{\sigma,\F,\varepsilon}\left(q^{-1+2\vartheta_\sigma+\varepsilon} \right).$$
For the special value of $\lmf=1$, we have
$$|\widetilde{\Mscr}(\sigma,1,\qmf)|=|\mathscr{G}_1(\qmf,\Psi_1,\Psi_1)|\leqslant q^{\vartheta}\mathscr{G}(\Psi)=q^{\vartheta}\left( ||\varphi||_{L^4}^4-||\varphi||_{L^2}^4 \right),$$
where $\Gscr(\Psi)$ is defined in \eqref{SpectralExpPhi_i} with $\Psi=|\varphi|^2$ and the last equality comes from Remark \ref{Dependance}.

Now recalling definitions \eqref{DefinitionMoment}, \eqref{ValueC} and the fact that $\Lcal(\pi,\qmf)=1$ when $\pi$ has conductor $\qmf$ (c.f. \eqref{ellv} and \eqref{Lcal}) , we have
$$\Mscr(\sigma,\qmf,1)=2\Lambda_\F(2)\sum_{\substack{\pi \ \mathrm{cuspidal} \\ \crm(\pi)=\qmf}}\frac{L(\pi\otimes\sigma\otimes\sigma,\tfrac{1}{2})}{\Lambda(\pi,\mathrm{Ad},1)}f(\pi_\infty)+\mathcal{U}(\sigma),$$
where $\mathcal{U}(\sigma)$ contains the Eisenstein expansion and the unramified representations appearing in the cuspidal part $\mathscr{C}(\sigma,\qmf,1)$. By $\eqref{kappav}$, we obtain
$$\mathcal{U}(\sigma)\ll_{\sigma,\F}q^{2\vartheta_\sigma}.$$
To complete the proof of Theorem \ref{Theorem4}, we just note that since $\qmf$ is prime, we have $\tfrac{\zeta_\qmf(1)}{\zeta_\qmf(2)}=1+q^{-1}.$
\end{proof}

%%%%%%%%%%%%%%%%%%%%%%%%%%%%%%%%%%%%%%%%%%%%
%%%%%%%%%%%%%%%%%%%%%%%%%%%%%%%%%%%%%%%%%%%%
%%%%%%%%%%%%%%%%%%%%%%%%%%%%%%%%%%%%%%%%%%%%
%%%%%%%%%%%%%%%%%%%%%%%%%%%%%%%%%%%%%%%%%%%%
%%%%%%%%%%%%%%%%%%%%%%%%%%%%%%%%%%%%%%%%%%%%
%%%%%%%%%%%%%%%%%%%%%%%%%%%%%%%%%%%%%%%%%%%%
%%%%%				SECTION SUBCONVEXITY
%%%%%%%%%%%%%%%%%%%%%%%%%%%%%%%%%%%%%%%%%%%%
%%%%%%%%%%%%%%%%%%%%%%%%%%%%%%%%%%%%%%%%%%%%
%%%%%%%%%%%%%%%%%%%%%%%%%%%%%%%%%%%%%%%%%%%%
%%%%%%%%%%%%%%%%%%%%%%%%%%%%%%%%%%%%%%%%%%%%
%%%%%%%%%%%%%%%%%%%%%%%%%%%%%%%%%%%%%%%%%%%%
%%%%%%%%%%%%%%%%%%%%%%%%%%%%%%%%%%%%%%%%%%%%
\section{Proof of Theorem \ref{Theorem2} and Corollary \ref{Corollary5}}\label{SectionSub}
Let $\qmf$ be a squarefree ideal of $\Ocal_\F$ and fix $\pi_0$ a cuspidal automorphic representation of $\PGL_2(\Abb_\F)$ with finite conductor $\qmf$. Let $\pi_1,\pi_2$ two unitary automorphic representations satisfying Hypothesis \ref{Hyp1}. We fix the test vectors $\varphi_{i}\in\pi_i$ as in the beginning of Section \ref{SectionEstimation}, i.e. we allow the infinite components $\varphi_{i,\infty}$ to have a certain degree of freedom (see Remark \ref{RemarkInfinite}).

%%%%%%%%%%%%%%%%%%%%%%%%%%%%%%%%%%%%%%%%%%%%
%%%%%%%%%%%%%%%%%%%%%%%%%%%%%%%%%%%%%%%%%%%%
%%%%% THE AMPLIFICATION METHOD
%%%%%%%%%%%%%%%%%%%%%%%%%%%%%%%%%%%%%%%%%%%%
%%%%%%%%%%%%%%%%%%%%%%%%%%%%%%%%%%%%%%%%%%%%
\subsection{The amplification method}
Let $q^{1/100}<L<q$ be a parameter that we will choose at the end. Given $\pi$ a unitary automorphic representation of conductor dividing $\qmf$, we choose as in \cite[Section 12]{blomerspectral} the following amplifier
$$\Acal(\pi):= \left(\sum_{\substack{\pmf \in \mathrm{Spec}(\Ocal_\F) \\ \Nscr(\pmf)\leqslant L \\ \pmf \nmid \qmf }}\lambda_\pi(\pmf)\bm{x}(\pmf)\right)^2+\left(\sum_{\substack{\pmf \in \mathrm{Spec}(\Ocal_\F) \\ \Nscr(\pmf)\leqslant L \\ \pmf \nmid \qmf }}\lambda_\pi(\pmf^2)\bm{x}(\pmf^2)\right)^2,$$
where $\bm{x}(\lmf)=\mathrm{sgn}(\lambda_\pi(\lmf))$.
Observe that
\begin{equation}\label{LowerBound2}
\Acal(\pi_0)\geqslant \frac{1}{2}\left(\sum_{\substack{\pmf \in \mathrm{Spec}(\Ocal_\F) \\ \Nscr(\pmf)\leqslant L \\ \pmf \nmid \qmf }}|\lambda_{\pi_0}(\pmf)|+|\lambda_\pi(\pmf^2)|\right)^2\gg_\F \frac{L^2}{(\log L)^2},
\end{equation}
by Landau Prime Ideal Theorem and the Hecke relation $\lambda_{\pi_0}(\pmf)^2=1+\lambda_{\pi_0}(\pmf^2).$ On the other hand, we have
\begin{equation}\label{DecompositionAmplifier}
\begin{split}
\Acal(\pi)=& \ \sum_{\substack{\pmf \in \mathrm{Spec}(\Ocal_\F) \\ \Nscr(\pmf)\leqslant L \\ \pmf \nmid \qmf }}(\bm{x}(\pmf)^2+\bm{x}(\pmf^2)^2)+\sum_{\substack{\pmf_1,\pmf_2  \\ \Nscr(\pmf_i)\leqslant L \\ \pmf_i \nmid \qmf }} \bm{x}(\pmf_1^2)\bm{x}(\pmf_2^2)\lambda_\pi(\pmf_1^2\pmf_2^2) \\ + & \ \sum_{\substack{\pmf_1,\pmf_2  \\ \Nscr(\pmf_i)\leqslant L \\ \pmf_i \nmid \qmf }} (\bm{x}(\pmf_1)\bm{x}(\pmf_2)+\delta_{\pmf_1=\pmf_2}\bm{x}(\pmf_1^2)\bm{x}(\pmf_2^2))\lambda_{\pi}(\pmf_1\pmf_2).
\end{split}
\end{equation}
Let $C$, $f(\pi_{0,\infty})$ be the quantity defined respectively in \eqref{ValueC} and \eqref{Definitionfinfty} and recall that $\Lcal(\pi_0,\qmf)=1$ for such a $\pi_0$ (c.f. \eqref{ellv} and \eqref{Lcal}). By positivity, we have
$$C\Acal(\pi_0)\frac{L(\pi_0\otimes\pi_1\otimes\pi_2,\tfrac{1}{2})}{\Lambda(\pi_0,\mathrm{Ad},1)}f(\pi_{0,\infty})\leqslant \Mscr_\Acal(\pi_1,\pi_2,\qmf),$$
with $\Mscr_\Acal(\pi_1,\pi_2,\qmf)$ as in \eqref{DefinitionMoment}, but with the amplifier $\Acal(\pi)$ instead of the Hecke eigenvalues in \eqref{CuspidalPart} and \eqref{ContinuousPart}. Using the lower bound \eqref{LowerBound2}, we get 
$$\frac{L\left( \pi_0\otimes\pi_1\otimes\pi_2,\tfrac{1}{2}\right)}{\Lambda(\pi_0,\mathrm{Ad},1)}f(\pi_{0,\infty}) \ll_{\varepsilon,\F} L^{-2+\varepsilon}\Mscr_\Acal(\pi_1,\pi_2,\qmf).$$
Expanding the amplifier as in \eqref{DecompositionAmplifier} and applying Theorem \ref{Theorem1} with $\lmf=1,\pmf_1\pmf_2$ or $\lmf=\pmf_1^2\pmf_2^2$ yields, provided $L<q^{1/4}$,
$$
\frac{L\left( \pi_0\otimes\pi_1\otimes\pi_2,\tfrac{1}{2}\right)}{\Lambda(\pi_0,\mathrm{Ad},1)}f(\pi_{0,\infty})  \ll_{\varepsilon,\F,\pi_1,\pi_2,\varphi_{1,\infty},\varphi_{2,\infty}} q^\varepsilon\left( qL^{-1}+q^{\frac{1}{2}+\vartheta}L^2\right),
$$
where we emphasize the dependance in the Archimedean components $\varphi_{i,\infty}$ in the above estimation. Finally, choosing $L=q^{\frac{1-2\vartheta}{6}}$ and we infer the final bound
\begin{equation}\label{HybridBound}
\frac{L\left( \pi_0\otimes\pi_1\otimes\pi_2,\tfrac{1}{2}\right)}{\Lambda(\pi_0,\mathrm{Ad},1)}f(\pi_{0,\infty})  \ll_{\varepsilon,\F,\pi_1,\pi_2,\varphi_{1,\infty},\varphi_{2,\infty}} q^{1-\frac{1-2\vartheta}{6}+\varepsilon}.
\end{equation}
%%%%%%%%%%%%%%%%%%%%%%%%%%%%%%%%%%%%%%%%%%%%
%%%%%%%%%%%%%%%%%%%%%%%%%%%%%%%%%%%%%%%%%%%%
%%%%%%%%%%%%%%%%%%%%%%%%%%%%%%%%%%%%%%%%%%%%
%%%%%				PROOF COROLLARY 5
%%%%%%%%%%%%%%%%%%%%%%%%%%%%%%%%%%%%%%%%%%%%
%%%%%%%%%%%%%%%%%%%%%%%%%%%%%%%%%%%%%%%%%%%%
%%%%%%%%%%%%%%%%%%%%%%%%%%%%%%%%%%%%%%%%%%%%
\begin{proof}[proof of Corollary \ref{Corollary5}] Assume that $\pi_1=\pi_2=:\sigma$ is cuspidal and set $\varphi_1=\varphi_2=: \varphi$. For each $v|\infty$, we choose $\varphi_v$ to be the vector of minimal weight in $\sigma_v$. By \eqref{HybridBound}, for any $\varepsilon>0$, there exists a positive constant $C(\sigma,\F,\varepsilon)$ such that for any cuspidal representation $\pi$ of conductor $\qmf$ ($\qmf$ is prime here), we have 
\begin{equation}\label{Hybrid1}
\frac{L\left( \pi\otimes\sigma\otimes\sigma,\tfrac{1}{2}\right)}{\Lambda(\pi,\mathrm{Ad},1)}f(\pi_{\infty}) \leqslant C(\sigma,\F,\varepsilon)q^{1-\frac{1-2\vartheta}{6}}.
\end{equation}
By Theorem \ref{Theorem4}, we have for $q$ large enough in terms of $\sigma,\F,\epsilon$,
$$
\frac{1}{q}\sum_{\substack{\pi \ \mathrm{cuspidal} \\ \crm(\pi)=\qmf}}\frac{L\left( \pi\otimes\sigma\otimes\sigma,\tfrac{1}{2}\right)}{\Lambda(\pi,\mathrm{Ad},1)}f(\pi_{\infty}) \geqslant \frac{\Delta_\F \Lambda(\sigma,\mathrm{Ad},1)^2}{\Lambda_\F(2)V_\F}.
$$
Hence the result follows from the bound \eqref{Hybrid1}, the factorization \eqref{Sym} and the choice of the constant
$$\lambda(\sigma,\F,\varepsilon):= \frac{\Delta_\F \Lambda(\sigma,\mathrm{Ad},1)^2}{\Lambda_\F(2)V_\F C(\sigma,\F,\varepsilon)}.$$
\end{proof}

%%%%%%%%%%%%%%%%%%%%%%%%%%%%%%%%%%%%%%%%%%%%
%%%%%%%%%%%%%%%%%%%%%%%%%%%%%%%%%%%%%%%%%%%%
%%%%%%%%%%%%%%%%%%%%%%%%%%%%%%%%%%%%%%%%%%%%
%%%					PROOF THEOREM 2
%%%%%%%%%%%%%%%%%%%%%%%%%%%%%%%%%%%%%%%%%%%%
%%%%%%%%%%%%%%%%%%%%%%%%%%%%%%%%%%%%%%%%%%%%
%%%%%%%%%%%%%%%%%%%%%%%%%%%%%%%%%%%%%%%%%%%%
\begin{proof}[proof of Theorem \ref{Theorem2}] Let $\pi$ be cuspidal with conductor $\qmf$ and assume here that $\pi_1$ and $\pi_2$ satisfy Hypothesis \ref{Hyp}. In consequences, we can choose test vectors $\varphi_{i,\infty}\in\pi_{i,\infty}$, depending on the Archimedean data $\pi_\infty$ such that \eqref{LowerBound} is satisfied. Observe that in this case the dependence in $\varphi_{i,\infty}$ in \eqref{HybridBound} is replaced by a dependence in $\pi_\infty$ (see Remark \ref{RemarkInfty}). Finally, using \eqref{BoundAdjoint} for the adjoint $L$-function at $s=1$, \eqref{HybridBound} transforms into
$$L\left( \pi\otimes\pi_1\otimes\pi_2,\tfrac{1}{2}\right) \ll_{\varepsilon,\F,\pi_1,\pi_2,\pi_\infty}q^{1-\frac{1-2\vartheta}{6}+\varepsilon},$$
which gives the desired subconvexity bound in Theorem \ref{Theorem2}.
\end{proof}

%%%%%%%%%%%%%%%%%%%%%%%%%%%%%%%%%%%%%%%%%%%%
%%%%%%%%%%%%%%%%%%%%%%%%%%%%%%%%%%%%%%%%%%%%
%%%%%%%%%%%%%%%%%%%%%%%%%%%%%%%%%%%%%%%%%%%%
%%%%%%%%%%%%%%%%%%%%%%%%%%%%%%%%%%%%%%%%%%%%
%%%%%					SECTION : LOCAL COMPUTATIONS
%%%%%%%%%%%%%%%%%%%%%%%%%%%%%%%%%%%%%%%%%%%%
%%%%%%%%%%%%%%%%%%%%%%%%%%%%%%%%%%%%%%%%%%%%
%%%%%%%%%%%%%%%%%%%%%%%%%%%%%%%%%%%%%%%%%%%%
%%%%%%%%%%%%%%%%%%%%%%%%%%%%%%%%%%%%%%%%%%%%
\section{A Local Computation}\label{SectionLocal}
The goal of this section is to compute explicitly the local factor \eqref{LocalPeriod} appearing in Proposition \ref{PropositionDegenerate}. We thus fix $k$ a non-Archimedean local field of characteristic zero with ring of integers $\Ocal$, maximal ideal $\mmf$, uniformizer $\varpi$ (i.e. a generator of the maximal ideal $\mmf$) and with residue field of size $q$. 

Let $\pi_1,\pi_2$ be generic irreducible admissible unitarizable representations of $\PGL_2(k)$ and write $\Krm=\GL_2(\Ocal)$. We assume that the representations $\pi_1,\pi_2$ are unramified and tempered. On each $\pi_i$, we fix an inner product $\langle \cdot,\cdot\rangle_i$ together with an equivariant isometrical map $\pi_i\rightarrow \Wcal(\pi_i)$ and the Whittaker models are equipped with the inner product \eqref{NormalizedInnerProduct}. 
We choose $\Krm$-invariant vectors $\varphi_i\in\pi_i$ with norm one with respect $\langle \cdot,\cdot\rangle_i$ and with associated Whittaker functions $\Wrm_i$ ($\Wrm_i(1)=1$ with this normalization). 

We let $f\in 1\boxplus 1$ be the section appearing as local constituant of the global section defined in \eqref{DefinitionFlat}. For small $\varepsilon>0$ and for $s,t\in\Cbb$ with $0\leqslant |s|+|t|\leqslant \varepsilon$ and $n\geqslant 1$, we want to evaluate the following integral
\begin{equation}\label{DefinitionZetaIntegral}
\Zcal(\pi_1,\pi_2,n;t,s) :=\int_{\Nrm(k)\setminus\PGL_2(k)}\Wrm_1\Wrm_2^{\mmf^n}f(t)^{\mmf^n}f(s)dg.
\end{equation}
To simplify notations, we will only treat the case $t=s=0$ and leave the general case to the reader; this only affects the final bound by a factor $q^{n\varepsilon}$. We thus write $\Zcal(\pi_1,\pi_2,n;0,0)=\Zcal(\pi_1,\pi_2,n)$. 

\vspace{0.1cm}

Using Iwasawa decomposition, the measure \eqref{HaarMeasure} and the definition of the section $f$ leads to 
$$\frac{\Zcal(\pi_1,\pi_2,n)}{\zeta_k(1)}= \int_{\Krm}\Fcal(k,n)dk,$$
with 
\begin{equation}\label{DefinitionFcal}
\Fcal(k,n)=f^{\mmf^n}(k)\int_{k^\times}\Wrm_1(a(y)k)\Wrm_2^{\mmf^n}(a(y)k)d^\times y.
\end{equation}
Oberve that the function $k\mapsto \Fcal(k,n)$ is left invariant by the subgroups 
\begin{equation}\label{GroupsInvariance}
\Nrm(k)\cap \Krm, \ \ \mathbf{A}(k)\cap \Krm, \ \ \Zrm(k)\cap \Krm.
\end{equation}
Moreover, since $\Fcal(k,n)$ is also right invariant under $\Krm_0(\varpi^n)$, we may decompose the $\Krm$-integral as follows :
\begin{equation}\label{DecompositionK-integral}
\int_{\Krm}\Fcal(k,n)dk = |\Krm_0(\varpi^n)|\Fcal(1,n)+\sum_{i=1}^n \int_{\Krm_0(\varpi^{i-1})\setminus \Krm_0(\varpi^i)}\Fcal(k,n)dk.
\end{equation}
We now have the following lemma which is a generalization of \cite[Lemma 11.6]{sparse} :
%%%%%%%%%%%%%%%%%%%%%%%%%%%%%%%%%%%%%%%%%%%%
%%%%%				LEMMA K-INTEGRAL
%%%%%%%%%%%%%%%%%%%%%%%%%%%%%%%%%%%%%%%%%%%%
\begin{lemme}\label{LemmaFcal} For any $i=1,...,n$ and $k\in \Krm_0(\varpi^{i-1})\setminus \Krm_0(\varpi^i)$, we have
$$\Fcal(k,n)=\left\{ \begin{array}{lcl}
\Fcal\left(\left(\begin{matrix}
 & -1 \\ 1 &
\end{matrix} \right),n\right) & \ifm & i=1 \\ 
 & & \\
 \Fcal\left(\left(\begin{matrix}
 1&  \\ \varpi^{i-1} & 1
\end{matrix} \right),n\right) & \ifm & 1<i\leqslant n.
\end{array}\right.$$
\end{lemme}
\begin{proof}
The strategy is to use the invariance properties of $\Fcal(k,n)$ under the subgroups \eqref{GroupsInvariance}. Up to multiplying by an element of $\Zrm(k)\cap \Krm$, we can assume that our matrix is of the form $k=\left(\begin{smallmatrix} a & b \\ \varpi^{i-1} & d\end{smallmatrix}\right)$. Multiplying on the left by $n(-a)$, we get $$n(-a)k=\left(\begin{matrix} a(1-\varpi^{i-1}) & b-ad \\ \varpi^{i-1} & d \end{matrix}\right)=: k_1.$$
We now need to distinguish the case $i=1$ from $i>1$. In the case $i=1$, we have $b-ad \in \Ocal^\times$, so multiplying $k_1$ on the left by $a((b-ad)^{-1})$, we obtain $k_2= \left(\begin{smallmatrix}  & 1 \\ 1 & d \end{smallmatrix}\right).$ We conclude by $w=k_2\left(\begin{smallmatrix}1 & d \\  & -1 \end{smallmatrix}\right)$. 

We use the same strategy if $i>1$. Here $a(1-\varpi^{-1})$ is a unit, so $a(a(1-\varpi^{i-1})^{-1})k_1 = \left(\begin{smallmatrix} 1 & \alpha \\ \varpi^{i-1} & d\end{smallmatrix} \right)=: k_2$ with $\alpha \in \Ocal.$ We continue by setting $$k_3=k_2\left(\begin{matrix} 1 & \alpha \\ & -1 \end{matrix}\right)=\left( \begin{matrix} 1 & \\ \varpi^{i-1} & \alpha\varpi^{i-1}-d\end{matrix}\right).$$
Observe that necessarily $a,d\in\Ocal^\times$, so that $\alpha\varpi^{i-1}-d\in\Ocal^\times$ and we conclude by multiplying on the right by the diagonal matrix having $1$ and the inverse of $\alpha\varpi^{i-1}-d$ in the low entry.
\end{proof}
%%%%%%%%%%%%%%%%%%%%%%%%%%%%%%%%%%%%%%%%%%%%
%%%%%				END LEMMA
%%%%%%%%%%%%%%%%%%%%%%%%%%%%%%%%%%%%%%%%%%%%
Write $k_0(\varpi^i)=\left( \begin{smallmatrix} 1 & \\ \varpi^i & 1\end{smallmatrix}\right)$. The function $f^{\mmf^n}$ can be easily evaluated for different choices of $k$. For this, we just need to observe that if $f$ is constructed from the characteristic function of the lattice $\Ocal^2$, then the right translate satisfies 
\begin{equation}\label{IdentitySection2}
f^{\mmf^n}=q^{n/2}f_{\Psi^{\mmf^n}} 
\end{equation}
with $\Psi^{\mmf^n}$ the characteristic function of the lattice $\mmf^n\times\Ocal$. We also get from the Definition \eqref{DefinitionFlat}
\begin{equation}\label{IdentitySection1}
\frac{f_{\Psi^{\mmf^n}}(k)}{\zeta_k(1)}= \left\{ \begin{array}{lcl}
1 & \ifm & k=1 \\ 
 & & \\
q^{i-n} & \ifm & k=k_0(\varpi^i), \ i=1,...,n-1 \\ 
 & & \\ 
 q^{-n} & \ifm & k=w. 
\end{array}\right.
\end{equation}
For the function $\Wrm_2$, we first have by right $\Krm$-invariance\begin{equation}\label{Identityw}
\Wrm_2^{\mmf^n}(a(y)w)=\Wrm_2(a(y)wa(\varpi^{-n}))=\Wrm_2(a(y\varpi^{n})w)=\Wrm_2(a(y\varpi^{n})).
\end{equation}
Moreover, looking at \cite[Lemma 2.13]{saha} inspires the factorization   
$$
\left(\begin{matrix}
1 & \\ \varpi^i & 1
\end{matrix}\right)= \left(\begin{matrix}
\varpi^i & \\ & \varpi^i
\end{matrix}\right)\left(\begin{matrix}
1 & \varpi^{-i} \\ & 1
\end{matrix}\right)
\left(\begin{matrix}
\varpi^{-2i} & \\ & 1
\end{matrix}\right)\left(\begin{matrix}
& -1 \\ 1 &
\end{matrix}\right)
\left(\begin{matrix}
1 & \varpi^{-i} \\ & 1
\end{matrix}\right),
$$
which follows from the Bruhat decomposition. Noting that $\Wrm_2^{\mmf^n}$ is right invariant by the subgroup $\Nrm(\varpi^{-n}\Ocal)$, it follows that for every $i=1,...,n-1$,
\begin{equation}\label{Facto}
\begin{split}
\Wrm_2^{\mmf^n}(a(y)k_0(\varpi^i))=  & \ \Wrm_2^{\mmf^n}(a(y)n(\varpi^{-i})a(\varpi^{-2i})wn(\varpi^{-i})) \\ 
 = & \ \psi(y\varpi^{-i})\Wrm_2^{\mmf^n}(a(y)a(\varpi^{-2i})w) \\ = & \ \psi(y\varpi^{-i})\Wrm_2^{\mmf^{2i-n}}(a(y)),
\end{split}
\end{equation}
where $\psi$ is the non-trivial additive and unramified character of $k$ associated to the Whittaker model of $\pi_2$. 

Define $k_0(\varpi^i)=w$ if $i=0$. Inserting Lemma \ref{LemmaFcal} in decomposition \eqref{DecompositionK-integral} and using successively \eqref{IdentitySection1}, \eqref{IdentitySection2} \eqref{Identityw}, \eqref{Facto} and the fact that $\Wrm_i(\varpi^r)=0$ if $r<0$ yields the following relation
\begin{equation}
q^{n-1}(q+1)\frac{\int_\Krm \Fcal(k,n)dk}{\zeta_k(1)q^{n/2}}= \sum_{i=0}^n \alpha_i \mathcal{I}(i,n),
\end{equation}
with 
\begin{equation}\label{DefinitionIcal}
\mathcal{I}(i,n)=\int_{k^\times}\psi(y\varpi^{-i})\Wrm_1(a(y))\Wrm_2^{\mmf^{2i-n}} d^\times y
\end{equation}
and $\alpha_i=1$ if $i=0$ or $i=n$ and $\zeta_k(1)^{-1}$ otherwise. These integrals can be estimated very easily as follows : we use first the fact that the $\Wrm_i$ are right invariant by $\Arm(k)\cap \Krm$, so that
$$\mathcal{I}(i,n)=\sum_{r\geqslant 0}\Wrm_1(a(\varpi^r))\Wrm_2\left(a\left(\varpi^{r+n-2i}\right)\right)\int_{u\in\Ocal^\times}\psi(u\varpi^{r-i})d^\times u$$
For the integral over the units, we explicit the Haar measure on $d^\times u=\zeta_k(1)\frac{du}{|\cdot|}$ on $k^\times$ and we use \cite[Proposition 1.6.5]{goldfeld} to find
$$\int_{u\in\Ocal^\times} \psi(u\varpi^{r-i})d^\times = \left\{ \begin{array}{lll}
1 & \ifm & r\geqslant i \\ 
 & & \\
-q^{r-i}\zeta_k(1) & \ifm & r=i-1 \\ 
 & & \\
 0 & \mathrm{else}. 
\end{array}\right.$$
Furthermore, it is a well-known result (see \cite[Chapter 4]{bump} for example) that if the $\pi_i$'s are tempered, we have for $r\geqslant 0$
$$|\Wrm_i(\varpi^r)|\leqslant (r+1)q^{-\frac{r}{2}}.$$
Therefore, recalling that $\Wrm_i(\varpi^r)=0$ for $r<0$, we obtain 
\begin{equation}\label{EstimationIcal}
\begin{split}
|\mathcal{I}(i,n)|\leqslant & \ q^{-n/2}\sum_{r\geqslant i}(r+1)(r+n-2i+1)q^{i-r} \\
 + & \delta_{1\leqslant i \leqslant n-1}q^{-n/2}\zeta_k(1)i(n-i)q^{-n/2} \\ 
  \ll & \ n(n-i)q^{-n/2}.
  \end{split}
\end{equation}
Finally, inserting \eqref{EstimationIcal} in \eqref{DecompositionK-integral} gives the following proposition :
\begin{proposition}\label{PropositionLocal} Let $\pi_1,\pi_2$ be tempered and unramified representations of $\PGL_2(k)$ and for $n\in\Nbb$, let $\Zcal(\pi_1,\pi_2,n;t,s)$ be the local zeta integral defined in \eqref{DefinitionZetaIntegral}. Then for any sufficiently small $\varepsilon>0$ and $t,s\in\Cbb$ with $0\leqslant |s|+|t|\leqslant \varepsilon$, we have
$$\Zcal(\pi_1,\pi_2,n;t,s) \ll n^2q^{n(-1+\varepsilon)}.$$
\end{proposition}

%%\printbibliography
%
\bibliography{Fourth}
%%\addcontentsline{toc}{section}{References}
\bibliographystyle{plain}
\end{document}